\documentclass{amsart}

\usepackage{amssymb}
\usepackage{amsmath}
\usepackage{hyperref}
\usepackage{color}
\usepackage[T1]{fontenc}
\usepackage{tikz} 
\usetikzlibrary{matrix,calc,arrows,decorations.markings,arrows.new}
\usepackage[labelfont=rm]{caption}
\usepackage{subcaption}

\newtheorem{thm}{Theorem}[section]

\newtheorem{prop}{Proposition}[section]
\newtheorem{cor}{Corollary}[section]
\newtheorem{lema}{Lemma}[section]

\theoremstyle{remark}
\newtheorem{obs}{Remark}[section]
\theoremstyle{remark}
\newtheorem{dfn}{Definition}[section]
\newtheorem{ntc}{Notation}[section]
\newtheorem{ejm}{Example}[section]

\newtheorem{clm}{Claim}
\numberwithin{equation}{section}

\makeatletter
\let\c@lema\c@thm
\let\c@prop\c@thm
\let\c@conj\c@thm
\let\c@cor\c@thm
\let\c@obs\c@thm
\let\c@dfn\c@thm
\let\c@ntc\c@thm
\let\c@ejm\c@thm
\makeatother

\def\makeautorefname#1#2{\expandafter\def\csname#1autorefname\endcsname{#2}}

\makeautorefname{lema}{Lemma}%
\makeautorefname{obs}{Remark}%
\makeautorefname{thm}{Theorem}%
\makeautorefname{dfn}{Definition}%
\makeautorefname{ejm}{Example}%
\makeautorefname{subfigure}{Figure}%
\makeautorefname{prop}{Proposition}%
\makeautorefname{cor}{Corollary}%
\makeautorefname{section}{Section}%

\newcommand\enet[1]{\renewcommand\theenumi{#1}
\renewcommand\labelenumi{\theenumi}}

\numberwithin{equation}{section}

\DeclareMathOperator{\spec}{Spec}

\DeclareMathOperator{\supp}{Supp}
\DeclareMathOperator{\ord}{ord}
\DeclareMathOperator{\np}{NP}
\DeclareMathOperator{\nr}{NR}
\DeclareMathOperator{\chr}{char}

\date{\today}

\title[Dicritical divisors, atypical fibres, special pencils and polynomials]{High-school algebra of the theory of dicritical divisors:
atypical fibres for special pencils and polynomials}
\author[E. Artal]{E. Artal Bartolo}
\address{IUMA\\
Departamento de Matem{\'a}ticas, Facultad de Ciencias, Universidad de Zaragoza,
c/ Pedro Cerbuna 12, 50009 Zaragoza, SPAIN}
\email{artal@unizar.es}

\author{I. Luengo}
\author[A. Melle]{A. Melle-Hern{\'a}ndez}
\address{ICMAT (CSIC-UAM-UC3M-UCM) \\
Departamento de {\'A}lgebra, 
Facultad de Ciencias Matem{\'a}ticas, Universidad Complutense, 28040 Madrid,  SPAIN}
\email{iluengo@mat.ucm.es,amelle@mat.ucm.es}
\thanks{}



\begin{document}

\maketitle

\begin{quote}
\it\small In this work we got a revival of our discussions about dicriticals with Ram. Dedicated
to the memory of S.S.~Abhyankar.
\end{quote}

\begin{abstract}
In this work we deal with dicritical divisors, curvettes and polynomials.
These objects have been one of the main research interests of S.S. Abhyankar during
his last years. In this work we provide  
some elementary proofs of some S.S. Abhyankar and I.~Luengo results for dicriticals
in the framework of formal power series. Based on these ideas we give 
a constructive way to find the atypical fibres of a special pencil
and give bounds for its number, which are sharper than the existing ones.
Finally, we answer a question of J.~Gwo{\'z}dziewicz finding polynomials
that reach his bound.
\end{abstract}

\section*{Introduction}

The study of the topology and geometry of polynomial maps is of great interest in Affine Algebraic
Geometry, for instance for the cancellation problem or affine exotic spaces. 
The  Jacobian problem is one of the main open problems in this area. 
Recently the local theory of algebraic 
dicritical divisors and curvettes has been developed (\cite{AL1,aba:13,aba:14}) and  
applied to get some control on the fibers of a Jacobian pair. Dicritical divisors have been studied by S.S.~Abhyankar either alone, \cite{ab:10b,ab:11a,ab:11b,ab:12,ab:13a,ab:13b}, 
or with co-authors, \cite{aba:13,aba:14,AH1,AH2,AH3,AL1,AL2}. He has developed an algebraic
theory which starts from the geometric intuition coming from analytic geometry and extends
the result to the more general setting: starting from $\mathbb{C}\{x,y\}$ he developed
(with his collaborators) a general theory valid for general regular local rings.

In this work we want to apply this theory to the study of special pencils, i.e., 
elements of the quotient field of a regular ring whose denominator is a power
of a regular element of the ring. The fundamental  reason to study  these pencils
is that they appear naturally when working with polynomial maps at infinity.
Moreover, the strategy to study these pencils is through the  resolution of the base points of the pencil where dicriticals appear
in a natural way. With their algebraic techniques, several results about dicriticals
are proved in \cite{AL1,AL2}: the restriction of the pull-back of the pencil to
each dicritical is a polynomial, dicriticals are in one-to-one correspondence with the irreducible 
factors of the pencil, see \S\ref{sec-kdt} for details.

The core of the paper is to provide elementary algebraic proofs, 
valid also in positive characteristic,
for  rings of power series over a field by high-school algebra methods
following  the mathematical philosophy  of S.S~Abhyankar.
In order to achieve the proof, we proceed with a variation of the Newton-Puiseux process
realized by birational transformations, see also~\cite{cn:11} for similar approaches. Using Newton 
polygon techniques we describe a finite recursive
argument which presents in an explicit case a toric resolution of the pencil which is 
combinatorially much less complex than the resolution via standard blow-ups or quadratic transformations.
With this method, the dicritical divisors are in bijection with some edges of a sequence of
Newton polygons, from which we keep two important data: a $1$-variable polynomial
coming from the edge and a positive integer which is related to a quotient singularity
coming from a toric blowing-up.

We will apply these techniques in order to improve some bounds for the number of atypical values 
of special pencils given by J.~Gwo{\'z}diewicz in~\cite{gw:13}.

\smallskip
\noindent\textbf{Theorem 1.1.} (\cite{gw:13})\begingroup\em
Let $f(x, y),  l(x, y)\in\mathbb{C}\{x, y\}$, $f(0, 0) = l(0, 0) = 0$, be convergent power series
without common factor. Assume that the curve $l(x, y) = 0$ is smooth and that the curve
$f(x, y) = 0$ has $d$ components counted without multiplicities. Then, the pencil 
$f(x, y)-tl(x, y)^M = 0$, where $M$ is a positive integer, has at most $d$ nonzero atypical values.
\endgroup
\smallskip

Our main result provides a more accurately defined  bound for the number of atypical values
for a special pencil which is given by the sum of
the number of dicriticals plus the number of non-zero roots of the derivatives of the
polynomials associated to the dicriticals, see Theorem~\ref{thm:bound}.  
Moreover, this result
is true for formal power series over algebraically closed fields without
restrictions on the characteristic (except a mild separability hypothesis),
following Abhyankar's style. Example~\ref{ej-bound-sp} shows that our bound is sharp.

This local bound is also extended to the polynomial setting, see also \cite{JT}.
Since at  each base point at infinity the polynomial defines a local  special pencil then
the number of atypical values at infinity
is bounded  by the sum of the corresponding  local bounds we got in  Theorem~\ref{thm:bound}. 
Therefore, as a consequence, an algebraic proof of the next Theorem is given.
 
\smallskip
\noindent\textbf{Theorem 1.2.} (\cite{gw:13})\begingroup\em
Assume that the complex algebraic curve $f(x, y) = 0$  has $n$ branches at
infinity. Then the polynomial $f$ has at most $n$ critical values at infinity 
different from $0$.
\endgroup
\smallskip

We also provide examples  showing  that our bound  is also  sharper  than  the one of
\cite[Theorem 1.2]{gw:13}.
Notice that Gwo{\'z}diewicz's result is in the same spirit as the following Moh's Theorem~\cite{moh:74} 
as quoted by Ephraim's version~\cite{ep:83}.

\smallskip
\noindent\textbf{Theorem 2.2.} (\cite{ep:83})\begingroup\em
Assume that the complex algebraic curve $f(x, y) = 0$ has only one branch
at infinity. Then $f$ has no critical values at infinity. In particular, all curves $f(x, y) = t$
for $t\in\mathbb{C}$ are equisingular at infinity.
\endgroup
\smallskip

As T.T. Moh pointed out in~\cite{moh:74},  S.S.  Abhyankar gave another proof of this result 
by applying \cite[(3.4)]{am:73}.  
 
The number of branches at infinity  is related with the Jacobian problem: 
\begin{quote}\em
if $f_1,f_2\in \mathbb{K}[x,y]$,  $\chr(\mathbb{K})=0$, is a Jacobian pair, i.e.
its  Jacobian determinant is equal to $1$, then 
$\mathbb{K}[f_1,f_2]= \mathbb{K}[x,y]$. 
\end{quote}
T.T.~Moh remarks  in~\cite{moh:74}  that the following Engel's statement 
was a main tool in W. Engel's attempted proof of the
Jacobian conjecture, see \cite{engel}: 
\begin{quote}\em
For a special member of the pencil $f (x, y) + c = 0$, the number
of branches at infinity cannot be greater than the corresponding number for
the general member.
\end{quote}
In 1971 S.S.~Abyhankar found a counterexample
to Engel's statement.

Abhyankar and Moh, see e.g \cite{ab:77} for details, translated the Jacobian condition 
into conditions  on the resulting special expansions getting the following 
result: 

\smallskip
\noindent\textbf{The Two Point Theorem.} (\cite{ab:77}) \begingroup\em
If $f_1$ and $f_2$ is a Jacobian pair, then $f_1$ and $f_2$ have at most two points at infinity.
Moreover, it can be deduced that if the Jacobian condition  implies that $f_1$ and $f_2$ have at most one point at 
infinity then the Jacobian problem has an affirmative answer.
\endgroup
\smallskip

In fact if $f_1$ and $f_2\in \mathbb{K}[x,y]$ is a Jacobian pair with two points at infinity 
it follows from H.~{\.Z}o{\l}{\k a}dek  in   \cite{zl:08} that
 $f_1$ and $f_2$ have some  common dicriticals. In fact, not all the dicritical components
can be in common because in such a case  the degree of the polynomial map from $\mathbb{C}^2$
to $\mathbb{C}^2$ vanishes, hence
the Jacobian is identically zero (private communication to the authors of Pierrette Cassou-Nogu{\`e}s).

As we explain in \S\ref{sec:poli}, the conditions to reach this number
of branches at infinity are quite involved (in particular Moh-Ephraim result shows
that it is not possible when there is only one branch at infinity).
The last part of \S\ref{sec:poli} is devoted to construct two examples.
Example~\ref{ej-bound-sp-polynomial} is the polynomial version
of Example~\ref{ej-bound-sp}. 
Example~\ref{ejm-cotaep} answers positively the following  question proposed by
J.~Gwo{\'z}diewicz~\cite{gw:13}.

\smallskip
\noindent\textbf{Question.} \begingroup\em
Does there exist a polynomial $f(x, y)$ with $n$ nonzero critical values at infinity 
such that the curve $f(x, y) = 0$ has $n$ branches at infinity?
\endgroup

Example~\ref{ejm-cotaep} is a polynomial where the generic fiber has two branches at infinity.
Following a referee's comment we provide in Example~\ref{ejm:rfr} a way to construct
such examples with an arbitrary number of branches at infinity for the generic fiber.

\section{Toric-Newton transforms of special meromorphic functions}

For convenience we work over  an algebraically closed field $\mathbb{K}$. 
Nevertheless, the results are valid over any field since it is well known that
one can get the resolution of the base points of a pencil over a finite extension  of the base field $\mathbb{K}$.
Let $R=\mathbb{K}[[x,y]]$ be the formal power series ring over~$\mathbb{K}$; note that
most of the results are also valid for convergent power series in case of complex numbers and some of them
will also be valid for more general (almost complete) two-dimensional local rings (without restriction
on the characteristic and 
even in mixed characteristic) especially if they have \emph{analytical} properties,
see~\cite{aba:14}. 
Following Abhyankar we will study 
 regular local rings contained in $L$ (the fraction field of $R$) and dominating $R$
 though we will replace these rings by their completion for simplicity. We will denote $M(R)$
 the maximal ideal of $R$.

A formal power series $p(x,y)\in R$ can be evaluated at the only closed point $0\in\spec R$, giving 
an element $p(0,0)\in \mathbb{K}$. For an element $r(x,y):=\frac{p(x,y)}{q(x,y)}$ the evaluation at~$0$
can be defined on $\mathbb{P}^1_\mathbb{K}=\mathbb{K}\cup\{\infty\}$, with one important exception.
If $p,q\in M(R)$ are coprime, then $r(0)$ is not defined, it is \emph{undetermined}.
It is also useful to treat $r$ as the \emph{pencil} of curves $\{C_t: p=t q\}$, for $t\in\mathbb{K}\cup\{\infty\}$
having~$0$ as base point.

It is well-known that one can eliminate this indetermination via a birational map $\pi:S\to\spec(R)$,
which is the composition of a sequence of closed points blow-ups, also called quadratic transformations,
such that $\pi^*(r):S\to\mathbb{P}^1_\mathbb{K}$ is a well
defined morphism. This means that from  the point of view of pencils,
the strict transforms of the curves $C_t$ are disjoint.

Let $E=\pi^{-1}(0)$ be the exceptional divisor of the map $\pi$, with irreducible components $E_1,\dots,E_s$.
A divisor $E_i\subset E$ is called \emph{dicritical} (or some authors called them \emph{horizontal}) 
if $\pi^*(r) |_{E_i}$ is not a constant map, that is 
$\pi^*(r) (E_i)= \mathbb{P}^1_\mathbb{K}$.

Let us define 
$$
P(x,y,T):=p(x,y)-T q(x,y)=\sum_{i,j} A_{i,j} x^i y^j\in \mathbb{K}(T)[[x,y]]\ (T \text{ an indeterminate).} 
$$
We have two main interests: 
To study the curve $\tilde{C}$ given by $P\in\mathbb{K}(T)[[x,y]]$ and to study
the curves $C_t=\{P(x,y,t)=0\}$ for $t\in \mathbb{K}$, both generic and atypical.

\begin{dfn}
The \emph{Newton polygon} $\np(r)$ of~$r$ is the Newton polygon of  $P\in\mathbb{K}(T)[[x,y]]$
i.e. the compact faces of the convex closure of $\nr(P):=\supp(P) +\mathbb{N}^2\subset \mathbb{N}^2\subset\mathbb{R}^2$.
\end{dfn}

We are interested in giving several algebraic characterizations of 
dicritical divisors in a particular class of pencils, specially important for polynomial maps.

\begin{dfn} A meromorphic germ $r\in L$ (or its corresponding pencil) is called \emph{special} if  
$r(x,y)=\frac{p(x,y)}{x^c U(x,y)}$ for some local parameters $x,y\in R$, $c>0$ and
a unit $U(x,y)\in\mathbb{K}[[x,y]]$ (we always assume that $x$ does not divide $p(x,y)$).
\end{dfn}

\begin{obs}
Since $x$ does not divide $p$, the $y$-order~$d$ of $p(x,y)$ is well-defined, i.e.
the unique positive integer such that $p(0,y)$ is a series of order~$d$.
\end{obs}

\begin{ejm}\label{ej:1}
The pencil 
$$
p_x(x,z,T)=(x^3-z^5)^2-x^6+x(x^5-z^2)^5
+5 x z^7 \left(x  -\frac{3}{4}z^2\right)-T z^{11}
$$
is special in  $\mathbb{K}(T)[[z,x]]$.
\end{ejm}

These pencils are called Ephraim pencils in~\cite{gw:13}, based on~\cite{ep:83}.
It was shown in \cite[Theorem~A]{AL1}  that for a special pencil~$r$, 
the restriction of the pull-back $\pi^*(r)$ to any dicritical divisor  
is a polynomial, for arbitrary two dimensional local regular
rings, not necessarily equicharacteristic. In this paper an
elementary proof of this result for $R=\mathbb{K}[[x,y]]$ is given;
the tools used in the proof
detect the so-called atypical fibers of the pencil which are also studied in this work. 

From now on we assume that $r\in L$ is special.
We are going to give a recursive method to solve a special pencil~$r$ by means of toric transformations 
and translations associated to $\np(r)$.

We introduce some notation. Fix an edge $\ell$ of $\np(r)$
which is contained in the line $n x+m y=\omega$ ($m,n\in\mathbb{N}$ coprime).
We denote by  $\omega_\ell$ the weight $\omega_\ell(i,j):=n i +m j$. 
This edge supports a $\omega_\ell$-quasi-homogeneous polynomial of degree~$\omega$
\begin{equation}\label{poly-toric}
P_\ell(x,y)=\sum_{\omega_\ell(i,j)=\omega} A_{i,j} x^i y^j=x^u y^v q_\ell(x^m,y^n),
\end{equation}
where $q_\ell(s_1,s_2)\in\mathbb{K}[T][s_1,s_2]$ is a homogeneous polynomial of degree~$d_\ell$ with at least two monomials
and coprime with $s_1 s_2$.
Note that
$$
P(x,y,T)=P_\ell(x,y)+\text{ monomials with }\omega_{\ell}\text{-degree }>\omega.
$$
The coefficients of $q_\ell(s_1,s_2)$ are in $\mathbb{K}$ with only one eventual exception:
if $v=0$ and $u=c$, i.e., the vertex $(c,0)$ is in $\ell\subset\np(r)$.
Bezout identity allows to choose 
\begin{equation}\label{eq:bezout}
a,b\in\mathbb{Z}_{>0}\text{ such that }b n-a m=1. 
\end{equation}

\begin{ntc}
The coprime weights $(n,m)$ will be denoted if necessary as $(n_\ell,m_\ell)$;
we will refer to $n$ as the \emph{v-ratio} and $m$ as the \emph{h-ratio}
of the edge~$\ell$.
\end{ntc}

The following concept appears also in \cite{cn:11}.

\begin{dfn} An edge $\ell$  of $\np(r)$ is called a \emph{dicritical edge} if $(c,0)$ is a vertex of  $\ell$
\end{dfn}

\begin{obs}\label{obs:trans}
We assume that if $(n,m)=(1,m)$ then $P_\ell$ is not proportional to $(y-Ax^m)^e$, $A\in\mathbb{K}$.
If it is the case, the change of variables $y=y_1+A x^m$
makes  the edge $\ell$ disappear.
The polygon $\np(r)$ has at most one dicritical edge.
\end{obs}

\begin{ejm}\label{ej:2}
Let us consider $p_x$ as in Example~\ref{ej:1}. Its Newton polygon is
in Figure~\ref{fig:newtonpxz1}. There is only one edge~$\ell$
and $P_\ell=(x^3-z^5)^2$ ($x$ plays the role of $y$, we keep these variables for further use in
\S\ref{sec:poli}). The edge is not dicritical.
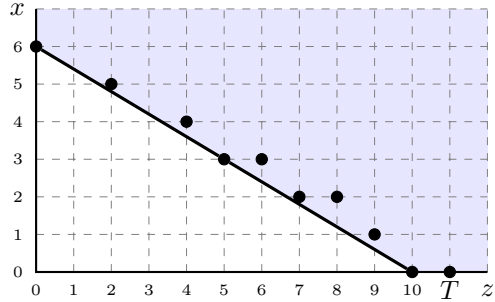
\begin{figure}[ht]
\begin{center}
\begin{tikzpicture}[scale=.5,xscale=1,vertice/.style={draw,circle,fill,minimum size=0.15cm,inner sep=0}]
\draw[fill, color=blue!10!white] (0,7)--(0,6)--(10,0)--(12,0)--(12,7);
\draw[help lines,step=1,dashed] (0,0) grid (12,7);
\draw[thick] (12,0)--(0,0)--(0,7);
\foreach \x in {0,...,6}
 \node[anchor=east] at (-.1,\x) {\tiny{\x}};
\foreach \y in {0,...,10}
 \node[anchor=north] at (\y,-.1) {\tiny{\y}};
\node at (-.5,7) {$x$};
\node at (12,-.5) {$z$};
\node[vertice] (a00) at (0,6) {};
\node[vertice] (a10) at (5,3) {};
\node[vertice] (a65) at (10,0) {};
\node[vertice] (a33) at (11,0) {};
\node[vertice] (a11) at (9,1) {};
\node[vertice] (a22) at (7,2) {};
\node[below] at (a33) {$T$};
\foreach \x/\xtext in {2,...,5}
\node[vertice] at (12-2*\x,\x) {};
\draw[,very thick] (a00)--(a10)--(a65);
\end{tikzpicture}
\caption{Newton polygon of $p_x$}
\label{fig:newtonpxz1}
\end{center}
\end{figure}
\end{ejm}

\begin{prop}\label{phim}
Assume that $\ell$ is not a dicritical edge.
The monomial transformation
$$
\varphi_M(x_1,y_1) :=(x_1^n y_1^a, x_1^m y_1^b), \quad \text{see \eqref{eq:bezout},}
$$
is birational (i.e. it is a composition of quadratic transformations) and
the polynomial $P_\ell$ is transformed as 
\begin{gather*}
P_\ell(x_1^n y_1^a, x_1^m y_1^b)=\beta x_1^{\omega} y_1^{a u+b v+ a m d_\ell} q_\ell(1,y_1).
\end{gather*}
\end{prop}

\begin{proof}
Note that
\begin{gather*}
P_\ell(x_1^n y_1^a, x_1^m y_1^b)=
\beta x_1^{n u+m v} y_1^{a u+b v} q_\ell(x_1^{m n} y_1^{a m},x_1^{m n} y_1^{b n})=
\beta x_1^{\omega} y_1^{a u+b v+ a m d_\ell} q_\ell(1,y_1).
\end{gather*}
We use that $\omega=n u+m v+m n d_\ell$, $bn-am=1$,  and the fact that $q_\ell$
is homogeneous of degree~$d_\ell$.
\end{proof}

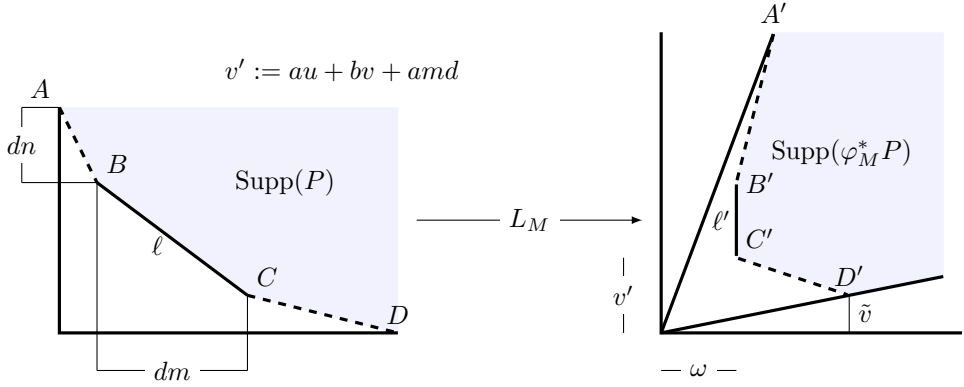
\begin{figure}[ht]
\begin{center}
\begin{tikzpicture}[scale=.5,xscale=1,vertice/.style={draw,circle,fill,minimum size=0.15cm,inner sep=0}]
\draw[fill, color=blue!5!white] (0,6)--(1,4)--(5,1)--(9,0)--(9,6);
\draw[very thick] (0,6)--(0,0)--(9,0);
\node[above left] at (0,6) {$A$};
\node[above right] at (1,4) {$B$};
\node[above right] at (5,1) {$C$};
\node[above] at (9,0) {$D$};
\node[below left] at (3,2.8) {$\ell$};
\node at (6,4) {$\supp(P)$};
\draw[very thick,dashed] (0,6)--(1,4);
\draw[,very thick] (1,4)--(5,1);
\draw[very thick,dashed] (5,1)--(9,0);
\draw (1,-1)--(5,-1) node[midway,fill=white] {$d m$};
\draw (-1,4)--(-1,6) node[midway,fill=white] {$d n$};
\draw (1,-1)--(1,4);
\draw (5,-1)--(5,1);
\draw (-1,6)--(0,6);
\draw (-1,4)--(1,4);
\draw[-latex] (9.5,3)--(15.5,3)node[midway,fill=white] {$L_M$};
\begin{scope}[xshift=16cm]
\draw[fill, color=blue!5!white] (3,8)--(2,4)--(2,2)--(5,1)--(7.5,1.5)--(7.5,8);
\draw[very thick] (0,8)--(0,0)--(8,0);
\node[inner sep=0] (A) at (3,8) {};
\node[above] at (A) {$A'$};
\node[inner sep=0]  (B) at (2,4) {};
\node[ right] at (B)  {$B'$};
\node[inner sep=0]  (C) at (2,2) {};
\node[above right] at (C)  {$C'$};
\node[inner sep=0]  (D) at (5,1) {};
\node[above] at (D) {$D'$};
\draw[dashed,very thick] (A)--(B);
\draw[very thick] (B)--(C);
\draw[dashed,very thick] (C)--(D);
\draw[very thick] (0,0)--(A);
\draw[very thick] (0,0)--(7.5,1.5);
\draw (0,-1)--(2,-1) node[midway,fill=white] {$\omega$};
\draw (-1,0)--(-1,2) node[midway,fill=white] {$v'$};
\draw (5,0)--(5,1) node[below right] {$\tilde{v}$};
\node at (-8.5,7) {$v':=a u +b v+  am d$};
\node at (1.6,3) {$\ell'$};
\node at (4.8,4.8) {$\supp(\varphi_M^* P)$};
\end{scope}
\end{tikzpicture}
\caption{$L_M$ for the edge $B C$.}
\label{fig:toric}
\end{center}
\end{figure}

This means that the image of the $\np(r)$ by the affinity 
$$
L_M:\mathbb{Z}^2\to\mathbb{Z}^2, 
\begin{pmatrix}
u\\v
\end{pmatrix}
\mapsto
\begin{pmatrix}
n&m\\a&b
\end{pmatrix}
\begin{pmatrix}
u\\v
\end{pmatrix}
$$ 
has a vertical edge
and $\supp(\varphi_M^* P)$, where $\varphi_M^* P=P(x_1^n y_1^a, x_1^m y_1^b,T)\in\mathbb{K}(T)[[x_1,y_1]]$,
is contained in $L_M(\nr(P))$, see~Figure~\ref{fig:toric}.
Let us factor
$$
q_\ell(s_1,s_2)=\beta\prod_{j=1}^e(s_2-\alpha_j s_1)^{m_j},\quad \beta,\alpha_j\in\mathbb{K}\setminus\{0\},\quad m_j>0,
\text{ i.e., }d_\ell=\sum_{j=1}^e m_j.
$$

\begin{dfn}
For $\ell$ a non-dicritical edge and $\alpha_j$ a root of $q_\ell(1,s)$, the 
\emph{toric-Newton transformation} associated
to $(\ell,\alpha_j)$ is the toric transformation $\varphi_M$
followed
by the translation $y_1=\bar{y}_1+ \alpha_j$.
\end{dfn}

\begin{dfn}
The \emph{strict transform} $P_{\ell,\alpha_j}(x_1,\bar{y}_1,T)$
of $P$ by the toric-Newton transformation associated
to $(\ell,\alpha_j)$ is
$$
P_{\ell,\alpha_j}(x_1,\bar{y}_1,T)=\frac{P(x_1^n 
(\bar{y}_1+ \alpha_j)^a, x_1^m (\bar{y}_1+ \alpha_j)^b,T)}%
{x_1^\omega (\bar{y}_1+ \alpha_j)^{\tilde{v}}}.
$$
where $\tilde v\leq a u+b v+ a m d$ is the minimum of
the powers of~$y_1$ which appear from the pull-back by $\varphi_M$.
\end{dfn}

\begin{ejm}\label{ej:3}
Let us study the strict transform for the toric-Newton transformation of Example~\ref{ej:1}.
The Newton polygon of this strict transform is shown in Figure~\ref{fig:newtonpxz2a}; the quasihomogeneous
polynomial is $(x_1-\frac{5}{2} z_1)^2$ and we are in the situation of Remark~\ref{obs:trans}. We perform the
translation and we obtain a special pencil whose Newton polygon, in Figure~\ref{fig:newtonpxz2b}
has only one edge and it is dicritical since the quasihomogeneous polynomial is $x_2^2-(T+\frac{5}{8}) z^3$.

\begin{figure}[ht]
\begin{center}
\begin{subfigure}[b]{.45\textwidth}
\begin{center}
\begin{tikzpicture}[scale=.5,xscale=1,vertice/.style={draw,circle,fill,minimum size=0.15cm,inner sep=0}]
\draw[fill, color=blue!10!white] (0,3)--(0,2)--(2,0)--(6,0)--(6,3);
\draw[help lines,step=1,dashed] (0,0) grid (6,3);
\draw[very thick] (6,0)--(0,0)--(0,3);
\foreach \x/\xtext in {0,...,2}
 \node[anchor=east] at (-.1,\x) {\tiny{\x}};
\foreach \y in {0,1,...,5}
 \node[anchor=north] at (\y,-.1) {\tiny{\y}};
\node[below] at (3,-.8) {$T$};
\node at (-.8,3) {$x_1$};
\node at (6,-.8) {$z_1$};
\draw[,very thick] (0,2)--(2,0);
\foreach \x/\xtext in {(5, 2), (5, 1), (4, 2), (5, 0), (4, 1), (3, 2), (4, 0), (3, 1), (2, 2),
(3, 0), (2, 1), (1, 2), (2, 0), (1, 1), (0, 2)}
\node[vertice] at \x {};
\end{tikzpicture}
\caption{}
\label{fig:newtonpxz2a}
\end{center}
\end{subfigure}
\begin{subfigure}[b]{.45\textwidth}
\begin{center}
\begin{tikzpicture}[scale=.5,xscale=1,vertice/.style={draw,circle,fill,minimum size=0.15cm,inner sep=0}]
\draw[fill, color=blue!10!white] (0,3)--(0,2)--(3,0)--(7,0)--(7,3);
\draw[help lines,step=1,dashed] (0,0) grid (7,3);
\draw[very thick] (0,3)--(0,0)--(7,0);
\foreach \x/\xtext in {0,...,2}
 \node[anchor=east] at (-.1,\x) {\tiny{\x}};
\foreach \y in {0,1,...,6}
 \node[anchor=north] at (\y,-.1) {\tiny{\y}};
\node[below] at (3,-.8) {$T$};
\node at (-.8,3) {$x_2$};
\node at (7,-.8) {$z_2$};
\draw[,very thick] (3,0)--(0,2);
\foreach \x/\xtext in {(7, 0), (6, 1), (5, 2), (6, 0), (5, 1), (4, 2), (5, 0), (4, 1), (3, 2),
(4, 0), (3, 1), (2, 2), (3, 0), (2, 1), (1, 2), (0, 2)}
\node[vertice] at \x {};
\end{tikzpicture}
\caption{}
\label{fig:newtonpxz2b}
\end{center}
\end{subfigure}
\caption{}
\label{fig:newtonpxz2}
\end{center}
\end{figure}
\end{ejm}

\begin{prop}\label{prop:toric-special}
The strict transform $P_{\ell,\alpha_j}(x_1,\bar{y}_1,T)$
is a special pencil in $\mathbb{K}[[x_1,\bar{y}_1]]$ such that
its $\bar{y}_1$-order is $m_j$.
\end{prop}

\begin{proof}
The part of the strict transform corresponding to $P_\ell$
is
\begin{gather*}
\beta 
\bar{y}_1^{m_j} \prod_{k\neq j}
(\bar{y}_1+\alpha_j-\alpha_k)^{m_k}.
\end{gather*}
The rest of the strict transform is divided by $x_1$.
The monomial $T x^c U(x,y)$ is transformed into
$$
T x^{n c-\omega} (\alpha_j+\bar{y}_1)^{a c-(a u+b v+ a m d)} U(x_1^n (\bar{y}_1+ \alpha_j)^a, x_1^m (\bar{y}_1+ \alpha_j)^b)
$$
and the result follows.
\end{proof}

We will study later what to do if $\ell$ is a dicritical edge.
Because of Proposition~\ref{prop:toric-special}, this process
can be also applied to the strict transforms of $P$
by the toric-Newton transformations.
\begin{dfn}
The \emph{toric-Newton process} of $P$ is the  sequence
of special pencils obtained by applying  toric-Newton transformations recursively.
The \emph{tree of Newton polygons} of $P$ is the family of all 
Newton polygons in the toric-Newton process.
An edge of such a Newton
polygon is called a \emph{dicritical} edge if it is at the bottom of the polygon
and the coefficient for $(*,0)$ depends on $T$.
\end{dfn}

\begin{prop}
The toric-Newton process is finite.
\end{prop}

\begin{proof}
Note that the $y$-order of the special pencils decreases unless we are in the situation
of Remark~\ref{obs:trans}. Since the pencils are special only a finite number of 
translations may arise until we reach the $T$-monomial. Note that while the term $Tx^c$
is not present in $NP(r)$ one is following the resolution (of one branch) of the fibre $p(x,y)=0$.
This means that  after a finite number of toric maps and translations we arrive to a point $Q$
where the branch is non-singular and eventually non-reduced. Then the local equation 
of the total transform of $P$ is $h^k(x_1,y_1) u(x_1,y_1)+Tx_1^{e_1}$ with $u(0,0)\ne 0$
and $h(x_1,y_1)=(y_1+\ldots).$  It is now clear we can make a change of coordinates
$y_1=\overline{y}+a_1x_1$ such that   $\overline{h}(x_1, \overline{y}))=\overline{y}+a_{e+1}x^{e+1}$
\end{proof}

\begin{obs}
Note that this is the case for the pencil in Example~\ref{ej:1}.
\end{obs}

\section{Dicritical edges}

Let us study now what happens with dicritical edges. We start
with a simple proof of \cite[Theorem~A]{AL1} when the regular local ring
is a formal power series ring.

\begin{prop}\label{prop:dicpol}
Let $P(x,y,T):=p(x,y)-T x^c U(x,y)$ be a special pencil, 
then at each dicritical divisor $E$ the function $\pi^*(r(x,y))_{|_E}$ is a polynomial.
\end{prop}

\begin{proof}
The previous process allows to resolve the base points of the pencil by toric maps 
and translations and moreover pencils arising at the process
are still special. Let us study what happens at a dicritical edge~$\ell$.
We keep the notation of~\eqref{poly-toric} and we get that
$$
q_\ell(1,s)=a_0 s^{d_\ell}+a_1 s^{d_\ell-1}+\dots+a_{d_\ell-1}  s-(T-a_{d_\ell})
$$
where $a_j\in\mathbb{K}$.
We denote again $\pi(x_1,y_1)=(x_1^n y_1^a,x_1^m y_1^b)$ the toric transformation associated
to~$\ell$. Then
$$
P_\ell(x_1^n y_1^a,x_1^m y_1^b)=x_1^{\omega} y_1^{\tilde{v}}
\left(q_\ell(1,y_1)+x_1 G(x_1,y_1)\right),
$$
and $x_1=0$ is the equation of $E$ and $G(x_1,y_1)$ is some series. 
Notice that 
$$\frac{\pi^*(p)}{\pi^*(x^c U(x,y))}=\frac{x_1^{\omega} y_1^{\tilde{v}}
\left(q_\ell(1,y_1)+x_1 G(x_1,y_1)\right)}{x_1^{\omega} y_1^{\tilde{v}}
(U(0,0)+x_1H(x_1,y_1))}=q_\ell(1,y_1)+x_1 G(x_1,y_1),$$
where $U(0,0)\neq 0$ and $H(x_1,y_1)$ is some series.
Restricting to $x_1=0$ we obtain the desired result. 

The computations above also prove that the corresponding polynomial map 
$q_E:E\to \mathbb{P}^1$, where $q_E(z):=q_\ell(1,z)-T$,
has degree $d_E:=d_\ell$.
\end{proof}

It is not hard to check  that the dicritical divisors of $r$ are in one to one correspondence with
the  dicritical edges
of $\np(r)$ and its transforms. We study now the toric-Newton transformations for dicritical edges. Note
that the toric part behaves as in the non-dicritical case, as shown in the proof
of Proposition~\ref{prop:dicpol}, but the translation part depends on the particular values
of~$t$. Moreover, separability properties of the polynomial $q_E(z)$ have a strong influence on
the behavior of the fibers of the pencil near the dicritical~$E$.

\begin{prop}\label{prop:typical}
Let $P(x,y,T)$ be a special pencil as above and let $E$ be a dicritical divisor 
of $r$ associated to a dicritical edge~$\ell$ of the toric-Newton process of~$P$.
Assume that $q_E(z)$ is a separable polynomial, i.e its derivative is not identically zero.

Let $A_E^*:=\{q_E(\alpha)\mid q_E'(\alpha)=0\}$ and let $t_{0,E}:=q_E(0)$. Then,
the strict transform of the germ of the curve $p(x,y)-t x^c U(x,y)$
contains exactly $d_E$ non-singular transversal curvettes meeting at $d_E$ distinct points of~$E$,
in the following cases:
\begin{enumerate}
\enet{\textup{(\arabic{enumi})}}
\item If $t\notin A_E^*$ and $t\neq t_{0,E}$.
\item If $t=t_{0,E}$, $t\notin A_E^*$ and $n=1$. 
\end{enumerate}
\end{prop}

\begin{proof}
We start with the first case. Since $t\notin A_E^*$ and the polynomial $q_E(z)$ is separable, 
we have that $\gcd(q_E(z)-t, q_E'(z))=1$ and
all the roots of $q_E(z)-t$ are simple roots, i.e.:
$$
q_E(z)-t=\prod_{i=1}^{d_E}(z-\alpha_i),
\quad \alpha_i\ne\alpha_j,\text{ if } i\neq j.
$$
Hence, the quasi-homogeneous polynomial associated to the edge~$\ell$ for the suitable
strict transform of $P(x,y,t)=0$ is
\begin{equation}\label{eq:nondeg}
\prod_{i=1}^{d_E}(y_1^n-\alpha_i x_1^m).
\end{equation}
Since $\alpha_i\neq 0$ and since $t\neq t_{0,E}$, all the above factors look similar.
Hence if we consider the (non trivial) translation $y_1=\bar{y}_1+\alpha_i$
\begin{equation}\label{eq:curvetas}
q_E(\bar{y}_1+\alpha_i)-t=b_0\bar{y}_1^{d_E}+b_1\bar{y}_1^{d_E-1}\ldots+ b_{d_E-1} \bar{y}_1,\quad
b_{d_E-1}\neq 0.
\end{equation}
If we compose the toric map of the proof
of Proposition~\ref{prop:dicpol} with the above translation, we obtain
then, up to terms of higher degree, that the strict transform
is written as 
$$
b_0\bar{y}_1^{d_E}+b_1\bar{y}_1^{d_E-1}\ldots+ b_{d_E-1} \bar{y}_1+x_1(\dots)
$$
and one gets $d_E$ non-singular curves intersecting transversally
the dicritical divisor $E: \{x_1=0\}$ at different points.

If $t=t_{0,E}$ is not a root of $q_E'(z)$ and $n=1$, though the Newton polygon is changing,
the factor corresponding to $\alpha_i=0$ is again a curvette.
\end{proof}

\begin{obs}
 With this method, along the exceptional dicritical  divisor there will be no base points of the 
pull-back of the pencil. By this process we get a log-canonical resolution (with quotient singularities)
the base points of the pencil. 
Since at each step we perform toric quadratic transformations we must be careful with the behavior
when no translation is needed.
\end{obs}

From  now on we assume that the map $q_E(z)$ is separable, i.e. either $\chr(\mathbb{K})=0$ or 
$\chr(\mathbb{K})=p$ and $q_E'(z)\neq 0$.

\begin{dfn}\label{dfn:atypical}
A value $t\in  \mathbb{K}$ is called \emph{a typical value
for $P(x,y,T)$ at $E$}  if the strict transform of the curve $P(t,x,y)$
has exactly $d_E$ non-singular branches  (curvettes) intersecting $E$ 
and is called  \emph{an atypical value for $P(x,y,T)$ at $E$}
otherwise.

If $t\in  \mathbb{K}$ is a  typical value
for $P(x,y,T)$ at all dicritical divisors  $E$
then $t\in  \mathbb{K}$ will be  called \emph{a typical value}
for $P(x,y,T)$, and \emph{an atypical} one otherwise.
\end{dfn}

\begin{ejm}\label{ej:4}
In Figure~\ref{fig:newtonpxz2b}, we have the Newton polygon of the unique dicritical edge
for $p_x$ in Example~\ref{ej:1}. If we fix $t=t_{0,E}=-\frac{5}{8}$, the vertex $(0,3)$
disappears. The corresponding Newton polygon is in Figure~\ref{fig:newtonpxz3}.
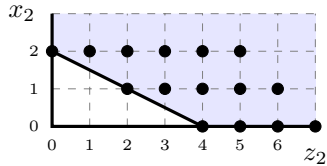
\begin{figure}[ht]
\begin{center}
\begin{tikzpicture}[scale=.5,xscale=1,vertice/.style={draw,circle,fill,minimum size=0.15cm,inner sep=0}]
\draw[fill, color=blue!10!white] (0,3)--(0,2)--(4,0)--(7,0)--(7,3);
\draw[help lines,step=1,dashed] (0,0) grid (7,3);
\draw[very thick] (0,3)--(0,0)--(7,0);
\foreach \x/\xtext in {0,...,2}
 \node[anchor=east] at (-.1,\x) {\tiny{\x}};
\foreach \y in {0,1,...,6}
 \node[anchor=north] at (\y,-.1) {\tiny{\y}};
\node at (-.8,3) {$x_2$};
\node at (7,-.8) {$z_2$};
\draw[,very thick] (0,2)--(2,1)--(4,0);
\foreach \x/\xtext in {(7, 0), (6, 1), (5, 2), (6, 0), (5, 1), (4, 2), (5, 0), (4, 1), (3, 2),
(4, 0), (3, 1), (2, 2), (2, 1), (1, 2), (0, 2)}
\node[vertice] at \x {};
\end{tikzpicture}
\caption{Final Newton polygon for the special fiber}
\label{fig:newtonpxz3}
\end{center}
\end{figure}
Since the general fiber is an ordinary cusp and for $t_{0,E}$ we have a tacnode,
we conclude that this value is atypical at~$E$.
\end{ejm}

\begin{obs}
In $\chr(\mathbb{K})=0$ this definition is equivalent to the standard definition, see,
for instance, the first definition 
in \cite[Section~3]{lw:97}. Note that the cases 
$(i)$ and $(iii)$ in that definition are not possible for special pencils:
$(i)$ in this case is only valid for $t_0=q_E(0)$ and $(iii)$ is not possible because
the first time ones gets  a dicritical divisor the linear system has no base points.  
\end{obs}

We are going to prove a sort of reciprocal of Proposition~\ref{prop:typical}.

\begin{thm}\label{thm:atypical}
Let $P(x,y,T)$ be a special pencil as in Proposition~\emph{\ref{prop:typical}}. Then
\begin{enumerate}
\enet{\textup{(\arabic{enumi})}}
\item\label{thm:atypical1} If $t\in A_E^*$ then $t$ is an atypical value for $P(x,y,T)$ at $E$.
\item\label{thm:atypical2} If $n>1$ ($n$ the v-ratio) then $t_{0,E}$ is atypical at $E$ 
regardless the value of $q_\ell'(t_{0,E})$.
\end{enumerate}
\end{thm}

\begin{obs}
From the interpretation of dicriticals of L{\^e}-Weber, the case $n>1$ corresponds exactly with
the dicriticals which admit a bamboo, see~\cite{lw:97}, which will be called \emph{dicriticals with bamboo}.
\end{obs}

\begin{proof} 
For the proof of \ref{thm:atypical1}, we follow the ideas in Proposition \ref{prop:typical}.
Let $\alpha_i$ be a multiple root of $q_E(s)-t$. In \eqref{eq:curvetas}, the condition
$b_{d_E-1}\neq 0$ fails and the corresponding point cannot be a curvette.

For \ref{thm:atypical2}, 
the Newton polygon of $P(x,y,t_{0,E})$ has a bottom edge which is non parallel to~$\ell$
and of height $n>1$, so there are some branches of this curve which do not meet~$E$,
see  Figure~\ref{fig:bambu} for a typical behavior of Newton polygons.
\end{proof}

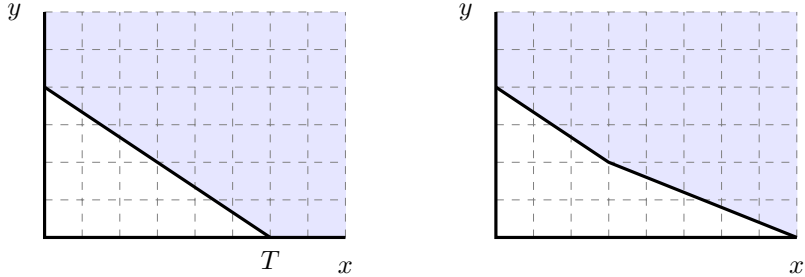
\begin{figure}[ht]
\begin{center}
\begin{tikzpicture}[scale=.5,xscale=1,vertice/.style={draw,circle,fill,minimum size=0.15cm,inner sep=0}]
\draw[fill, color=blue!10!white] (0,6)--(0,4)--(6,0)--(8,0)--(8,6);
\draw[help lines,step=1,dashed] (0,0) grid (8,6);
\draw[very thick] (0,6)--(0,0)--(8,0);
\node[below] at (6,-.1) {$T$};
\node at (-.8,6) {$y$};
\node at (8,-.8) {$x$};
\draw[,very thick] (0,4)--(3,2)--(6,0);
\begin{scope}[xshift=12cm]
\draw[fill, color=blue!10!white] (0,6)--(0,4)--(3,2)--(8,0)--(8,6);
\draw[help lines,step=1,dashed] (0,0) grid (8,6);
\draw[very thick] (0,6)--(0,0)--(8,0);
\node at (-.8,6) {$y$};
\node at (8,-.8) {$x$};
\draw[,very thick] (0,4)--(3,2)--(8,0);
\end{scope}
\end{tikzpicture}
\caption{Left-hand-side polygon for generic~$T$, right-hand-side for $t_{0,E}$.}
\label{fig:bambu}
\end{center}
\end{figure}

\begin{ejm} Let us describe some examples.

\begin{enumerate}
\item Consider the special pencil  $P(x,y,T)=y^4+y^2x^3+yx^7+x^{12}-Tx^6$,
see $\np(P)$ in Figure~\ref{fig:ej1}. 
The edge $\ell=[(0,4),(6,0)]$ is a dicritical edge such that
$P_\ell(x,y)= y^4+y^2x^3-Tx^6$,  $q_E(z)=z^2+z-T$ and $q_E(z)$ is separable.
Since the v-ratio $n$ equals $2>1$, $t=0$ is an atypical value, see its Newton polygon
in Figure~\ref{fig:ej2}.
On the other side $-\frac{1}{2}$ is the only root of $q_\ell'$ and then $t=-\frac{1}{4}$
is the other atypical value at $E$, see the Newton polygon 
after the toric-Newton transformation in Figure~\ref{fig:ej3}.
In this case a generic fibre has two branches at~$E$
while there are $3$ branches for $t=0$ and only one branch for $t=-\frac{1}{4}$. 

\begin{figure}[ht]
\begin{center}
\begin{subfigure}[b]{.3\textwidth}
\begin{tikzpicture}[xscale=.25,yscale=1,vertice/.style={draw,circle,fill,minimum size=0.15cm,inner sep=0}]
\draw[fill, color=blue!10!white] (0,5)--(0,4)--(6,0)--(13,0)--(13,5);
\draw[help lines,step=1,dashed] (0,0) grid (13,5);
\draw[very thick] (13,0)--(0,0)--(0,5);
\foreach \x/\xtext in {0,...,4}
 \node[anchor=east] at (-.1,\x) {\tiny{\x}};
\foreach \y in {0,1,...,12}
 \node[anchor=north] at (\y,-.1) {\tiny{\y}};
\node[below] at (6,-.3) {$T$};
\node at (-.8,5) {$y$};
\node at (13,-.2) {$x$};
\draw[,very thick] (0,4)--(6,0);
\foreach \x/\xtext in {(0,4), (3,2), (7, 1), (12, 0), (6,0)}
\node[vertice] at \x {};
\end{tikzpicture}
\caption{}
\label{fig:ej1}
\end{subfigure}
\begin{subfigure}[b]{.3\textwidth}
\begin{tikzpicture}[xscale=.25,yscale=1,vertice/.style={draw,circle,fill,minimum size=0.15cm,inner sep=0}]
\draw[fill, color=blue!10!white] (0,5)--(0,4)--(3,2)--(7,1)--(12,0)--(13,0)--(13,5);
\draw[help lines,step=1,dashed] (0,0) grid (13,5);
\draw[very thick] (13,0)--(0,0)--(0,5);
\foreach \x/\xtext in {0,...,4}
 \node[anchor=east] at (-.1,\x) {\tiny{\x}};
\foreach \y in {0,1,...,12}
 \node[anchor=north] at (\y,-.1) {\tiny{\y}};
\node[below] at (6,-.3) {\phantom{$T$}};
\node at (-.8,5) {$y$};
\node at (13,-.2) {$x$};
\draw[very thick] (0,4)--(3,2)--(7,1)--(12,0);
\draw[dashed] (3,2)--(12,0);
\foreach \x/\xtext in {(0,4), (3,2), (7, 1), (12, 0)}
\node[vertice] at \x {};
\end{tikzpicture}
\caption{}
\label{fig:ej2}
\end{subfigure}
\begin{subfigure}[b]{.3\textwidth}
\begin{tikzpicture}[xscale=.25,yscale=1,vertice/.style={draw,circle,fill,minimum size=0.15cm,inner sep=0}]
\draw[fill, color=blue!10!white] (0,6)--(0,2)--(5,0)--(13,0)--(13,6);
\draw[help lines,step=1,dashed] (0,0) grid (13,6);
\draw[very thick] (13,0)--(0,0)--(0,6);
\foreach \x/\xtext in {0,...,5}
 \node[anchor=east] at (-.1,\x) {\tiny{\x}};
\foreach \y in {0,1,...,12}
 \node[anchor=north] at (\y,-.1) {\tiny{\y}};
\node[below] at (6,-.3) {\phantom{$T$}};
\node at (-.5,6) {$y$};
\node at (13,-.2) {$x$};
\draw[,very thick] (0,2)--(5,0);
\foreach \x/\xtext in {(12, 6), (12, 5), (12, 4), (12, 3), (12, 2), (12, 1), (12, 0), (5, 3),
(5, 2), (5, 1), (5, 0), (0, 2)}
\node[vertice] at \x {};
\end{tikzpicture}
\caption{}
\label{fig:ej3}
\end{subfigure}
\caption{}
\label{fig:ej}
\end{center}
\end{figure}

\item  For the special pencil  $P(x,y,T)=y^3+y^2x-x^4-T x^3$
the edge $\ell=[(0,3),(3,0)]$ is dicritical and  $q_E(z)=z^3+z^2-T$,
see $\np(P)$ in Figure~\ref{fig:eja1}.
The derivative has two roots $0,-\frac{2}{3}$, and then $0,\frac{4}{27}$
are the atypical values. Since the v-ratio is~$1$,  $t=0$ is atypical only for being
a critical value of $q_E$, see its Newton polygon in Figure~\ref{fig:eja2}.
In order to study the fiber for $t=\frac{4}{27}$, we can check that
the quasi-homogeneous polynomial has one simple root and one double root.
It is enough to study what happens on the double root; instead of the toric-Newton
transformation we can do the change $y=y_1-\frac{2}{3}x$, and we obtain again
the Newton polygon of Figure~\ref{fig:eja2}.
All the typical fibres have 3 branches while the atypical ones have $2$  branches.

\begin{figure}[ht]
\begin{center}
\begin{subfigure}[b]{.45\textwidth}
\begin{center}
\begin{tikzpicture}[xscale=.5,yscale=.5,vertice/.style={draw,circle,fill,minimum size=0.15cm,inner sep=0}]
\draw[fill, color=blue!10!white] (0,4)--(0,3)--(3,0)--(5,0)--(5,4);
\draw[help lines,step=1,dashed] (0,0) grid (5,4);
\draw[very thick] (5,0)--(0,0)--(0,4);
\foreach \x/\xtext in {0,...,3}
 \node[anchor=east] at (-.1,\x) {\tiny{\x}};
\foreach \y in {0,1,...,4}
 \node[anchor=north] at (\y,-.1) {\tiny{\y}};
\node[below] at (3,-.6) {$T$};
\node at (-.5,4) {$y$};
\node at (5,-.5) {$x$};
\draw[,very thick] (0,3)--(3,0);
\foreach \x/\xtext in {(4, 0), (3, 0), (1, 2), (0, 3)}
\node[vertice] at \x {};
\end{tikzpicture}
\caption{}
\label{fig:eja1}
\end{center}
\end{subfigure}
\begin{subfigure}[b]{.45\textwidth}
\begin{center}
\begin{tikzpicture}[xscale=.5,yscale=.5,vertice/.style={draw,circle,fill,minimum size=0.15cm,inner sep=0}]
\draw[fill, color=blue!10!white] (0,4)--(0,3)--(1,2)--(4,0)--(5,0)--(5,4);
\draw[help lines,step=1,dashed] (0,0) grid (5,4);
\draw[very thick] (5,0)--(0,0)--(0,4);
\foreach \x/\xtext in {0,...,3}
 \node[anchor=east] at (-.1,\x) {\tiny{\x}};
\foreach \y in {0,1,...,4}
 \node[anchor=north] at (\y,-.1) {\tiny{\y}};
\node[below] at (3,-.6) {\phantom{$T$}};
\node at (-.5,4) {$y$};
\node at (5,-.5) {$x$};
\draw[,very thick] (0,3)--(1,2)--(4,0);
\foreach \x/\xtext in {(4, 0), (1, 2), (0, 3)}
\node[vertice] at \x {};
\end{tikzpicture}
\caption{}
\label{fig:eja2}
\end{center}
\end{subfigure}
\caption{}
\label{fig:eja}
\end{center}
\end{figure}

\item For the special pencil  $P(x,y,T)=y^3+yx^2-x^4-T x^3$, the value $t=0$ is typical
at the unique dicritical, even if the Newton polygons do not coincide, see Figure~\ref{fig:ejb}.    

\begin{figure}[ht]
\begin{center}
\begin{subfigure}[b]{.45\textwidth}
\begin{center}
\begin{tikzpicture}[xscale=.5,yscale=.5,vertice/.style={draw,circle,fill,minimum size=0.15cm,inner sep=0}]
\draw[fill, color=blue!10!white] (0,4)--(0,3)--(3,0)--(5,0)--(5,4);
\draw[help lines,step=1,dashed] (0,0) grid (5,4);
\draw[very thick] (5,0)--(0,0)--(0,4);
\foreach \x/\xtext in {0,...,3}
 \node[anchor=east] at (-.1,\x) {\tiny{\x}};
\foreach \y in {0,1,...,4}
 \node[anchor=north] at (\y,-.1) {\tiny{\y}};
\node[below] at (3,-.6) {$T$};
\node at (-.5,4) {$y$};
\node at (5,-.5) {$x$};
\draw[,very thick] (0,3)--(3,0);
\foreach \x/\xtext in {(4, 0), (3, 0), (2,1), (0, 3)}
\node[vertice] at \x {};
\end{tikzpicture}\caption{}
\label{fig:ejb1}
\end{center}
\end{subfigure}
\begin{subfigure}[b]{.45\textwidth}
\begin{center}
\begin{tikzpicture}[xscale=.5,yscale=.5,vertice/.style={draw,circle,fill,minimum size=0.15cm,inner sep=0}]
\draw[fill, color=blue!10!white] (0,4)--(0,3)--(2,1)--(4,0)--(5,0)--(5,4);
\draw[help lines,step=1,dashed] (0,0) grid (5,4);
\draw[very thick] (5,0)--(0,0)--(0,4);
\foreach \x/\xtext in {0,...,3}
 \node[anchor=east] at (-.1,\x) {\tiny{\x}};
\foreach \y in {0,1,...,4}
 \node[anchor=north] at (\y,-.1) {\tiny{\y}};
\node[below] at (3,-.6) {\phantom{$T$}};
\node at (-.5,4) {$y$};
\node at (5,-.5) {$x$};
\draw[,very thick] (0,3)--(2,1)--(4,0);
\foreach \x/\xtext in {(4, 0),  (2,1), (0, 3)}
\node[vertice] at \x {};
\end{tikzpicture}
\caption{}
\label{fig:ejb2}
\end{center}
\end{subfigure}
\caption{}
\label{fig:ejb}
\end{center}
\end{figure}

\end{enumerate}

\end{ejm}

\begin{obs}\label{obs:carp}
Note that Proposition \ref{prop:typical} and Theorem \ref{thm:atypical}
gives a complete characterization of atypical values of a special pencil in terms of 
the polynomials $q_E(z)$ if they are separable. 
In the inseparable case the atypical values cannot be computed  just from
$q_E$ as the following examples, in $\chr(\mathbb{K})=p$, show:

\begin{enumerate}
\enet{\rm(\alph{enumi})}
\item\label{nsep-a} $y^p+x^{p+1}-T x^{p}$, $t=0$ behaves as the other values of~$\mathbb{K}$.

\item\label{nsep-b} $y^p+y^2x^{p-1}+T x^p$, $t=0$ does not behave as the other values of~$\mathbb{K}$. 

\end{enumerate}

In both cases the generic members of the pencil have the singularity type of 
$y^p+x^{p+1}$. In particular it is not a curvette, as curvettes are smooth, 
and following our definition all values would
be atypical. A natural extension of our definition to the non-separable case
would imply that $t=0$ is typical for \ref{nsep-a} and atypical for \ref{nsep-b}.
See \cite{mctcw:01} for a more complete description of pencils in positive characteristic.
\end{obs}

In the separable case, we can recover algebraically the results of \cite{gw:13}.
More precisely it is possible to recover the number of atypical fibers only in terms
of the Newton polygons. The type of the atypical fibers needs the part \emph{behind}
the Newton polygons, but for the number, these Newton polygons are enough, compare
with Remark~\ref{obs:carp}.

We would like  to estimate the number of atypical values at a dicritical.
Let us collect the relevant information from the Newton process.
We have $E_1,\dots,E_r$ dicriticals coming from dicritical edges
$\ell^1,\dots,\ell^r$, each one carries a polynomial
$q_i(z):=q_{E_i}(z)$ of degree~$d_i$ and from the weight $\omega_{\ell^i}$
we keep the number $n_i$. The separability hypothesis asserts that
$q_i(z)$ is separable.

\begin{thm}\label{thm:bound}
Let $P(x,y,T)$ be a special pencil satisfying the separability hypothesis.
Let $E$ be a dicritical and let $n$ be its v-ratio. Let $A_E$ be the set $\{q_E(\alpha)\mid q_E'(\alpha)=0\}$.
Then, the set of atypical values for $P(x,y,T)$ at  $E$ is 
$$
\begin{cases}
A_E\cup\{q_E(0)\}&\text{ if }n>1\\
A_E&\text{ if }n=1
\end{cases}
$$
In particular, the number of atypical values for $P(x,y,T)$  at $E$ is at most
$$
M_E:=\#\{\text{non-zero roots of }q_E'\}+1\,,
$$
and the number of atypical values for $P(x,y,T)$ is at most $\sum\noindent_{E \text{dicritical }} M_E$.
\end{thm}

\begin{proof}
This is a direct consequence of Theorem~\ref{thm:atypical}.
\end{proof}

The following result is an easy consequnece of Theorem~\ref{thm:bound}.

\begin{cor}[Gwo{\'z}dziewicz~\cite{gw:13}]\label{cor:gw}
Let $P(x,y,T)$ be a special pencil.
\begin{enumerate}
\enet{\textup{(\arabic{enumi})}}
\item If $E$ is a dicritical divisor of degree~$d_E$, then there are at most $d_E$ atypical values at $E$.
\item If there is a value $t_0$ such that $C_{t_0}^{\text{\rm red}}$ has $r$ branches at a dicritical divisor $E$
then there are at most $r$ atypical values at~$E$ (besides eventually~$t_0$).
\item The number of atypical values of the pencil is bounded by $\min(\nu^{\text{\rm gen}},\nu^{\text{\rm min}}+1)$,
where $\nu^{\text{\rm gen}}$ is the number of branches of the generic value and
$\nu^{\text{\rm min}}$ is the minimal number of branches of the fibers.
\end{enumerate}
\end{cor}

\begin{obs}\label{obs:bound}
In order to reach the bound $\nu^{\text{\rm gen}}$, the following conditions must happen. For every dicritical~$E$,
one has $n>1$, $q_E'(t_{0,E})\neq 0$, $q_E'$ has simple roots, and these roots have distinct values by $q_E$.
Moreover, the sets of atypical values for each dicritical are pairwise disjoint.
\end{obs}

\begin{ejm}\label{ej-bound-sp}
 Let us consider the special pencil 
$$
P(x,y,T)=y^4-2x^2y^2+(y^2-x^2)yx^2+x^7-Tx^4
$$ 
which,
for all $t\in\mathbb{K}$  has 4 branches; the bound proposed in Corollary~\ref{cor:gw},
see~\cite{gw:13}, for the number of atypical values is at most 4.
Let us compute the bound of Theorem~\ref{thm:bound}. 
The unique edge $\ell$  of the Newton polygon is dicritical and for its dicritical~$E$ we have 
$q_E(z)=z^4-2z^2-T$. The roots of   $q_E'(z)$ are $\alpha=0,1,-1$, hence the bound equals~$3$. 
Since $q_E(0)=0$ and $q_E(1)=q_E(-1)=1$, there are exactly two atypical values, $t=0,1$.  
Figure~\ref{fig:ejc1} shows $\np(P(x,y,T))$, while Figure~\ref{fig:ejc2} shows $\np(p(x,y))$.
Note that Figure~\ref{fig:ejc2} shows also $\np(p(x,y\pm x)-1)$.
\end{ejm}

\begin{figure}[ht]
\begin{center}
\begin{subfigure}[b]{.45\textwidth}
\begin{center}
\begin{tikzpicture}[xscale=.5,yscale=.5,vertice/.style={draw,circle,fill,minimum size=0.15cm,inner sep=0}]
\draw[fill, color=blue!10!white] (0,5)--(0,4)--(4,0)--(8,0)--(8,5);
\draw[help lines,step=1,dashed] (0,0) grid (8,5);
\draw[very thick] (8,0)--(0,0)--(0,5);
\foreach \x/\xtext in {0,...,4}
 \node[anchor=east] at (-.1,\x) {\tiny{\x}};
\foreach \y in {0,1,...,7}
 \node[anchor=north] at (\y,-.1) {\tiny{\y}};
\node[below] at (4,-.6) {$T$};
\node at (-.5,5) {$y$};
\node at (8,-.5) {$x$};
\draw[,very thick] (0,4)--(2,2)--(4,0);
\foreach \x/\xtext in {(7, 0), (4, 1), (2, 3), (4, 0), (2, 2), (0, 4)}
\node[vertice] at \x {};
\end{tikzpicture}
\caption{}
\label{fig:ejc1}
\end{center}
\end{subfigure}
\begin{subfigure}[b]{.45\textwidth}
\begin{center}
\begin{tikzpicture}[xscale=.5,yscale=.5,vertice/.style={draw,circle,fill,minimum size=0.15cm,inner sep=0}]
\draw[fill, color=blue!10!white] (0,5)--(0,4)--(2,2)--(4,1)--(7,0)--(8,0)--(8,5);
\draw[help lines,step=1,dashed] (0,0) grid (8,5);
\draw[very thick] (8,0)--(0,0)--(0,5);
\foreach \x/\xtext in {0,...,4}
 \node[anchor=east] at (-.1,\x) {\tiny{\x}};
\foreach \y in {0,1,...,7}
 \node[anchor=north] at (\y,-.1) {\tiny{\y}};
\node[below] at (4,-.6) {\phantom{$T$}};
\node at (-.5,5) {$y$};
\node at (8,-.5) {$x$};
\draw[,very thick] (0,4)--(2,2)--(4,1)--(7,0);
\draw[dashed] (2,2)--(7,0);
\foreach \x/\xtext in {(7, 0), (4, 1), (2, 3), (2, 2),(0,4)}
\node[vertice] at \x {};
\end{tikzpicture}
\caption{}
\label{fig:ejc2}
\end{center}
\end{subfigure}
\caption{}
\label{fig:ejc}
\end{center}
\end{figure}

\section{\texorpdfstring{Factors of a special pencil over $\mathbb{K}(T)$}{Factors of a special pencil over 
K(T)}}
\label{sec-kdt}

Let us interpret a result of~\cite{AL2} in this language, always in the special
case of power series, namely that the
dicritical divisors of $r$ are in one-to-one correspondence with the factors of  $P(x,y,T)$ in $\mathbb{K}(T)[[x,y]]$.

Fix a dicritical edge and keep the notations of Proposition~\ref{prop:dicpol}.

\begin{prop}\label{prop-irr1}
Let $\ell$ be a dicritical edge of the $\np(r)$ corresponding to a dicritical divisor~$E$. Then there
exists an irreducible factor 
$Q_\ell(x,y,T)\in\mathbb{K}(T)[[x]][y]\subset\mathbb{K}(T)[[x,y]]$ 
of an element $P(x,y,T)$ such that 
its weighted initial form for $\omega_\ell$
equals $q_\ell(x^{m_\ell},y^{n_\ell})$.
\end{prop}

\begin{proof}
Note first that using Weierstra{\ss} Preparation Theorem, $P(x,y,T)$
can be decomposed as a product of a unit and a Weierstra{\ss} polynomial in~$y$
(recall that $P$ is $y$-regular of order~$d$).
We apply the version of Hensel's Lemma in \ref{sec:hensel} to this Weierstra{\ss} polynomial and the result follows.
\end{proof}

\begin{obs}
Instead of using Hensel's Lemma one can follow the ideas in \cite[Section~2]{aclm:09}
\end{obs}

If the dicritical edge $\ell $ is in another special pencil $r_1$ of the toric-Newton process
with coordinates $(x_1,\bar{y}_1)$,
then Proposition~\ref{prop-irr1} allow us to construct an irreducible factor
$\tilde{Q}_\ell(x_1,\bar{y}_1)$ of $r_1$ in $\mathbb{K}(T)[[x_1]][y_1]$; this factor
is $\bar{y}_1$-regular of order~$d_\ell$. Let us see the effect of the inverse
of the toric-Newton transformation in this element which produced $r_1$. 
The toric Newton transformation has two parts; 
the inverse of the translation is $\bar{y}_1\mapsto \bar{y}_1+\alpha_j=y_1$
while $\varphi_M^{-1}(x_1,y_1)=(x_1^b y_1^{-m},x_1^{-a} y_1^n)=(\bar{x},\bar{y})$.
Hence, the inverse
of the toric-Newton transformation will is
$$
(x_1,\bar{y}_1)\mapsto (x_1^b (\bar{y}_1+\alpha_j)^{-m},x_1^{-a} (\bar{y}_1+\alpha_j)^n)=(\bar{x},\bar{y}).
$$
Taking out denominators we obtain $\bar{Q}_\ell(\bar{x},\bar{y})$ which is a divisor of the 
special pencil $\bar{P}(\bar{x},\bar{y},T)$ at this level. It is not hard to see that
$\bar{Q}_\ell(\bar{x},\bar{y})$ is $\bar{y}$-regular of order~$n d_\ell$.
The contribution of this factor to the $\bar{y}$-degree is the expected one.
We
continue till we arrive to the first level; at each step the degree on the $y$-coordinate
is multiplied by the corrsponding v-ratio. The final pull-back $Q_\ell$ of $\tilde{Q}_\ell$ (taking out denominators)
to $\mathbb{K}(T)[[x,y]]$ is an irreducible factor of $P(x,y,T)$.

Let $\ell^1,  \ldots,\ell^s$ be the dicritical edges of the toric-Newton process.
For each dicritical edge $\ell^i$ we consider the sequence of
v-ratios $n_1^i,\dots,n_{h_i}^i$ ($h_i$ is the number of steps till $\ell^i$ appear) and its degree
$d_{\ell^i}$. The factor $Q_{\ell^i}$ has $y$-order 
$$
d_i:=d_{\ell_i}\cdot\prod_{j=1}^{h_i} n_{j}^i
$$
If $d=\ord_y(P)$, note that $d=\sum_{j=1}^ s d^j$ and  we conclude the next Theorem, see \cite{AL2} in more generality.

\begin{thm}
Let $P(x,y,T)$ be a special pencil. Then there is a one-to-one
correspondence between dicritical edges of the pencil and irreducible factors of $P\in\mathbb{K}[T][[x,y]]$.
By this correspondence to an edge $\ell^j$ we associate the factor
$Q_{\ell^j}$. 
\end{thm}

For typical $t\in\mathbb{K}$ the irreducible components of $\tilde{C}_t:=\{P(x,y,t)=0\}$, 
i.e. $\spec(R/(P(x,y,t)))$, are in one-to-one correspondence with the factors
of $P(x,y,T)$ in $\overline{\mathbb{K}(T)}[[x,y]]$ and the factors corresponding
to a given factor in $\mathbb{K}[T][[x,y]$ are the curvettes of the corresponding dicritical
(as many as the degree).

\section{Special pencils, polynomials and atypical fibers}\label{sec:poli}

In this section we recall the well-known relationship between special pencils and polynomials.
The polynomial $f(x,y)\in\mathbb{K}[x,y]$, $D:=\deg f$, defines a polynomial map 
$f:\mathbb{A}^2_{\mathbb{K}}\to\mathbb{A}^1_{\mathbb{K}}$, where $\mathbb{A}^j:=\mathbb{A}^j_{\mathbb{K}}$ 
is the affine space of dimension~$j$ over $\mathbb{K}$. We consider $(x,y)$ the affine coordinates
of $\mathbb{A}^2$ and $[X:Y:Z]$ the homogeneous coordinates of $\mathbb{P}^2:=\mathbb{P}^2_{\mathbb{K}}$
with the inclusion $(x,y)\hookrightarrow[x:y:1]$. Let us consider the rational extension of 
$f$ to a map $\tilde{f}:\mathbb{P}^2\dashrightarrow\mathbb{P}^1\equiv\mathbb{K}\cup\{\infty\}$.
If $f(x,y)=\sum_{j=0}^D f_j(x,y)$ is the decomposition in homogeneous components then
$$
B:=\{[u:v:0]\mid f_D(u,v)=0\}
$$
is the set of base points of~$\tilde{f}$.
At every base point $P_0\in B$ (at the line at infinity) the corresponding pencil  is an  special pencil. 

 Assume that $P_0:=[1:0:0]$ is one of these points.
In the affine chart $X\neq 0$ (with affine coordinates $y,z$) this map looks like
$$
\frac{f_y(y,z)}{z^D},\quad f_y(y,z):=z^D f\left(\frac{1}{z},\frac{y}{z}\right)
$$
and the fibers of $\tilde{f}$ near $P_0$ are of the form $f_y(y,z)-t z^D=0$, for $t\in\mathbb{K}\cup\{\infty\}$,
hence a special pencil.

By definition the dicriticals of the polynomial $f$ at infinity are the dicriticals
of the corresponding special pencils at all base points $P_0\in B$.
We define accordingly the atypical values at infinity at a dicritical of the polynomial~$f$,
see also~\cite{du:98}. 

In \cite{gw:13},  Gwo{\'z}dziewicz finds that the number of atypical values at infinity
of a polynomial is bounded above by the minimum of the two following numbers:
\begin{itemize}
\item The number $\nu_{\infty}^{\text{\rm gen}}$ of the branches at infinity of a generic fiber.
\item\label{numin} The number $\nu_{\infty}^{\text{\rm min}}+1$ where $\nu_\infty^{\text{\rm min}}$  is the minimal
number of branches at infinity for any fiber. 
\end{itemize}
Therefore an \emph{algebraic proof} of these results follows immediately from our algebraic proof of Corollary~\ref{cor:gw}.

In the same work, Gwo{\'z}dziewicz  asked if it is possible to reach the bound $\nu_\infty^{\text{\rm gen}}$
(or $\nu_{\infty}^{\text{\rm min}}+1$).
As we have observed in Remark~\ref{obs:bound}, to reach this bound imposes strong
conditions on the special pencils over all the dicriticals $E$:
\begin{itemize}
\item $n_E>1$.
\item $q_E'$ must have simple roots.
\item $q_E$ must pairwise \emph{separate} the values of $0$ and the roots of $q_E'$.
\item The sets of atypical values are disjoint for any pair of dicriticals. 
\end{itemize}
When we  deal with polynomials the last condition must be applied to any 
dicritical at infinity. Besides this difficulty the geometry of the polynomials
imposes more difficulties to find an example reaching the bound. 

Namely, no polynomial with only one dicritical reaches the bound. Assume for simplicity
that the polynomial is primitive. Then, the only dicritical is of degree~$1$, see e.g. \cite{ea:cre}.
Hence, by \cite{moh:83} all the fibers have only one branch at infinity
and by \cite{ep:83}, there is no atypical value at infinity.
It is not hard to find polynomials with two dicriticals $E_1,E_2$ both of multiplicity one but $n_i>1$.
These polynomials have two branches at infinity and one atypical value for each dicritical.
The problem is that most obvious examples satisfy that the set of atypical values is the same for both dicriticals.

\smallskip
\noindent\textbf{Gwo{\'z}dziewicz's question.} \begingroup\em
Does there exist a polynomial $f(x, y)$ with $n$ nonzero critical values at infinity 
such that the curve $f(x, y) = 0$ has $n$ branches at infinity?
\endgroup

\begin{ejm}\label{ejm-cotaep}
No polynomial of degree $\leq 10$ and two dicriticals reaches the bound.
The polynomial
$$
p(x,y)=x^{6} y^{5} - 5 x^{5} y^{4} + 10 x^{4} y^{3} - 2 x^{3} y^{3} - 10 x^{3}
y^{2} + 5 x^{2} y^{2} + 5 x^{2} y - \frac{15}{4} x y -  x + y
$$
does.
We will show that this polynomial $p(x,y)$ has two non-zero critical values at infinity and the curve 
$p(x,y)=0$ has two branches at infinity. 
This polynomial can be written as
$$
p(x,y)={\left(x^{3} y^{2} - 1\right)}^{2} y+{\left(x y - 1\right)}^{5} x-x^{6} y^{5}
+5 x y \left(x y -\frac{3}{4}\right).
$$
\begin{figure}[ht]
\begin{center}
\begin{subfigure}[b]{.45\textwidth}
\begin{center}
\begin{tikzpicture}[scale=.5,vertice/.style={draw,circle,fill,minimum size=0.1cm,inner sep=0}]
\draw[fill, color=blue!10!white] (0,0)--(1,0)--(6,5)--(0,1)--cycle;
\draw[help lines,step=1,dashed] (0,0) grid (6,5);
\draw[very thick] (0,6)--(0,0)--(7,0);
\foreach \x/\xtext in {0,...,6}
 \node[anchor=north] at (\x,-.1) {\x};
\foreach \y in {0,...,5}
 \node[anchor=east] at (-.1,\y) {\y};

\node[vertice] (a00) at (0,0) {};
\node[vertice] (a10) at (1,0) {};
\node[vertice] (a65) at (6,5) {};
\node[vertice] (a33) at (3,3) {};
\node[vertice] (a01) at (0,1) {};
\node[vertice] (a11) at (1,1) {};
\node[vertice] (a22) at (2,2) {};
\foreach \x/\xtext in {2,...,5}
\node[vertice] at (\x,\x-1) {};
\draw[,very thick] (a00)--(a10)--(a65)--(a01)--(a00);
\node[below] at (11,-.6) {\phantom{$t$}};
\end{tikzpicture}
\caption{Newton polygon of $p$}
\label{fig:newtonpxy}
\end{center}
\end{subfigure}
\begin{subfigure}[b]{.45\textwidth}
\begin{center}
\begin{tikzpicture}[scale=.5,xscale=1,vertice/.style={draw,circle,fill,minimum size=0.15cm,inner sep=0}]
\draw[fill, color=blue!10!white] ((0,7)--(0,5)--(10,0)--(12,0)--(12,7);
\draw[help lines,step=1,dashed] (0,0) grid (12,7);
\draw[very thick] (0,7)--(0,0)--(12,0);
\foreach \x/\xtext in {0,...,6}
 \node[anchor=east] at (-.1,\x) {\tiny{\x}};
\foreach \y in {0,1,...,11}
 \node[anchor=north] at (\y,-.1) {\tiny{\y}};
\node[below] at (11,-.6) {$t$};
\node at (-.8,7) {$y$};
\node at (12,-.8) {$z$};
\draw[,very thick] (0,5)--(10,0);
\foreach \x/\xtext in {(10, 1), (9, 1), (10, 0), (11, 0), (7, 2), (8, 1), (5, 3), (6, 2), (4,
3), (2, 4), (0, 5)}
\node[vertice] at \x {};
\end{tikzpicture}\caption{Newton polygon of $p_y$}
\label{fig:newtonpyz1}
\end{center}
\end{subfigure}
\caption{}
\end{center}
\end{figure}

In order to obtain the resolution of the polynomial we have to study the special
pencils located at the two points at infinity of $p$.
The first one is given by
$$
p_x(x,z)=(x^3-z^5)^2+\dots-t z^{11}
$$
and it is the one in Example~\ref{ej:1} (see also Example~\ref{ej:4}). We have seen that it has only one dicritical
which is of degree~one and v-ratio $2$. There is only one atypical value for this dicritical,
namely $t=-\frac{5}{8}$.

Let us study now the special pencil associated to the other point at infinity:
$$
p_y(y,z)=(y-z^2)^5+\dots-t z^{11}
$$

We are in the situation of Remark~\ref{obs:trans}, hence we
perform a translation as a change of variables,
$y=y_1+z_1^2$, $z=z_1$. In Figure~\ref{fig:newtonpyz2a} we see the new Newton polygon
where the coefficient of $z^{11}$ equals $-\left(t+\frac{3}{4}\right)$.
The Newton polygon for $t=-\frac{3}{4}$ is in Figure~\ref{fig:newtonpyz2b}. Hence, there is one atypical value
for this polynomial associated to this dicritical.

\begin{figure}[ht]
\begin{center}
\begin{subfigure}[b]{.48\textwidth}
\begin{tikzpicture}[scale=.5,xscale=.8,vertice/.style={draw,circle,fill,minimum size=0.15cm,inner sep=0}]
\draw[fill, color=blue!10!white]  (0,7)--(0,5)--(11,0)--(13,0)--(13,7);
\draw[help lines,step=1,dashed] (0,0) grid (13,7);
\draw[very thick] (0,7)--(0,0)--(13,0);
\foreach \x/\xtext in {0,...,6}
 \node[anchor=east] at (-.1,\x) {\tiny{\x}};
\foreach \y in {0,1,...,12}
 \node[anchor=north] at (\y,-.1) {\tiny{\y}};
\node[below] at (11,-.6) {$t$};
\node at (-.8,7) {$y_1$};
\node at (13,-.8) {$z_1$};
\draw[,very thick] (0,5)--(11,0);
\foreach \x/\xtext in {(12, 0), (10, 1), (11, 0), (9, 1), (7, 2), (5, 3), (0, 5)}
\node[vertice] at \x {};
\end{tikzpicture}
\caption{}
\label{fig:newtonpyz2a}
\end{subfigure}
\begin{subfigure}[b]{.48\textwidth}
\begin{tikzpicture}[scale=.5,xscale=.8,vertice/.style={draw,circle,fill,minimum size=0.15cm,inner sep=0}]
\draw[fill, color=blue!10!white] (0,6)--(0,5)--(9,1)--(12,0)--(13,0)--(13,6);
\draw[help lines,step=1,dashed] (0,0) grid (13,6);
\draw[very thick] (13,0)--(0,0)--(0,6);
\foreach \x/\xtext in {0,...,5}
 \node[anchor=east] at (-.1,\x) {\tiny{\x}};
\foreach \y in {0,1,...,12}
 \node[anchor=north] at (\y,-.1) {\tiny{\y}};
\node at (-.8,6) {$y_1$};
\node at (13,-.8) {$z_1$};
\draw[,very thick] (0,5)--(9,1)--(12,0);
\foreach \x/\xtext in {(12, 0), (10, 1), (9, 1), (7, 2), (5, 3), (0, 5)}
\node[vertice] at \x {};
\end{tikzpicture}
\caption{}
\label{fig:newtonpyz2b}
\end{subfigure}
\caption{}
\label{fig:newtonpyz2}
\end{center}
\end{figure}

Then, the two atypical values for each dicritical are different
and the polynomial~$p$ reaches the bound: as many non-zero atypical fibers at infinity as branches at infinity
for the fiber $p(x,y)=0$. The two atypical fibers at infinity have three branches.
The polynomial $p$ has only one (affine) singular fiber $p^{-1}(-\frac{20}{27})$ which has an ordinary double point
at $\left(-900,-\frac{4}{3375}\right)$.
\end{ejm}

\begin{figure}[ht]
\begin{center}
\begin{subfigure}[b]{.45\textwidth}
\begin{center}
\begin{tikzpicture}[scale=.5,xscale=1,vertice/.style={draw,circle,fill,minimum size=0.15cm,inner sep=0}]
\draw[fill, color=blue!10!white] (0,0)--(6,0)--(8,2)--(0,6);
\draw[help lines,step=1,dashed] (0,0) grid (9,7);
\draw[very thick] (0,7)--(0,0)--(9,0);
\foreach \x/\xtext in {0,...,6}
 \node[anchor=east] at (-.1,\x) {\tiny{\x}};
\foreach \y in {0,1,...,8}
 \node[anchor=north] at (\y,-.1) {\tiny{\y}};
\node at (-.5,7) {$y$};
\node at (9,-.5) {$x$};
\draw[,very thick] (0,0)--(6,0)--(8,2)--(0,6);
\foreach \x/\xtext in {(8, 2), (6, 3), (6, 2), (4, 4), (6, 1), (4, 3), (2, 5), (6, 0), (4, 2),
(2, 4), (0, 6), (4, 1), (2, 3), (0, 5), (4, 0), (2, 2), (0, 4), (2, 1),
(0, 3), (2, 0), (0, 1),(0,0)}
\node[vertice] at \x {};
\node[below] at (10,-.6) {\phantom{$t$}};
\end{tikzpicture}
\caption{Newton polygon of $f$}
\label{fig:newtonfxy}
\end{center}
\end{subfigure}
\begin{subfigure}[b]{.45\textwidth}
\begin{center}
\begin{tikzpicture}[scale=.5,xscale=1,vertice/.style={draw,circle,fill,minimum size=0.15cm,inner sep=0}]
\draw[fill, color=blue!10!white] (11,0)-- (4,0)--(0,2)--(0,7)--(11,7);
\draw[help lines,step=1,dashed] (0,0) grid (11,7);
\draw[very thick] (0,7)--(0,0)--(11,0);
\foreach \x/\xtext in {0,...,6}
 \node[anchor=east] at (-.1,\x) {\tiny{\x}};
\foreach \y in {0
,1,...,10}
 \node[anchor=north] at (\y,-.1) {\tiny{\y}};
\node[below] at (10,-.6) {$t$};
\node at (-.5,7) {$y$};
\node at (11,-.5) {$z$};
\draw[,very thick] (4,0)--(0,2);
\foreach \x/\xtext in {(9, 1), (7, 3), (6, 4), (5, 5), (4, 6), (8, 0), (7, 1), (6, 2), (5, 3),
(4, 4), (3, 5), (6, 0), (5, 1), (4, 2), (3, 3), (2, 4), (4, 0), (3, 1),
(2, 2), (1, 3), (0, 2),(10,0)}
\node[vertice] at \x {};
\end{tikzpicture}\caption{Newton polygon of $f_y$}
\label{fig:newtonfyz}
\end{center}
\end{subfigure}
\caption{}
\end{center}
\end{figure}

\begin{ejm}\label{ej-bound-sp-polynomial}
In the same way as in the local case, see Example \ref{ej-bound-sp}, the following polynomial shows that 
our bounds are better than the ones in \cite{gw:13}. Consider the  following polynomial of degree $10$ (see its Newton polygon
in Figure~\ref{fig:newtonfxy}): 
\begin{gather*}
f(x,y) = y^6-4(x^2+1)y^5+\left(12x^2+6x^4+\frac{41}{4}\right)y^4-\left(4x^6+\frac{25}{2}
+12x^4+\frac{99}{4}x^2\right)y^3\\
+\left(x^8+4x^6+\frac{75}{4}x^4+\frac{59}{4}x^2\right)y^2
+
\left(-\frac{17}{4}x^6+\frac{75}{4}x^2+4x^4+\frac{25}{4}\right)y\\
-\frac{25}{2}x^2-\frac{25}{4}x^6-\frac{71}{4}x^4.
 \end{gather*}
This polynomial has  two points at infinity,  that is $P_0=[1:0:0]$ and $P_1=[0:1:0] $ 
Thus the corresponding special pencil at $P_0$ is given by 
\begin{gather*}
 f_y(z,y)-Tz^{10}= y^6z^4-4y^5z^3-4y^5z^5+12y^4z^4+6y^4z^2+\frac{41}{4}y^4z^6-4y^3z-\frac{25}{2}y^3z^7\\
-12y^3z^3-\frac{99}{4}y^3z^5+4y^2z^2+\frac{75}{4}y^2z^4+\frac{59}{4}y^2z^6
+y^2-\frac{17}{4}yz^3\\
+\frac{75}{4}yz^7+4yz^5+\frac{25}{4}yz^9
-\frac{25}{2}z^8-\frac{25}{4}z^4-\frac{71}{4}z^6-Tz^{10}
\end{gather*}
Let us see that
this special pencil has  $2$ branches for all $t\in \mathbb{K}$ and it has two dicriticals $E_1$ and $E_2$ of  degree $1$.  
Its Newton polygon (see Figure~\ref{fig:newtonfyz}) has only one edge  $\ell$ 
which is not dicritical and such that 
$$
P_\ell=y^2-\frac{25}{4}z^4=\frac{(2 y-5 z^2) (2 y+5z^2)}{4}.
$$
Thus $q_\ell$ has degree 2 and two  simple roots $\pm\frac{5}{2}$. Making the toric-Newton transformation 
associated to each root $(\ell,\pm\frac{5}{2})$
one gets two dicriticals, each  one  of degree~1 (which are sections with no bamboo). Moreover,  these two
dicriticals have no atypical value associated. 

\begin{figure}[ht]
\begin{center}
\begin{subfigure}[b]{.4\textwidth}
\begin{center}
\begin{tikzpicture}[scale=.5,xscale=.8,vertice/.style={draw,circle,fill,minimum size=0.15cm,inner sep=0}]
\draw[fill, color=blue!10!white] (11,0)-- (4,0)--(0,8)--(0,9)--(11,9);
\draw[help lines,step=1,dashed] (0,0) grid (11,9);
\draw[very thick] (0,9)--(0,0)--(11,0);
\foreach \x/\xtext in {0,...,8}
 \node[anchor=east] at (-.1,\x) {\tiny{\x}};
\foreach \y in {0,1,...,10}
 \node[anchor=north] at (\y,-.1) {\tiny{\y}};
\node[below] at (10,-.6) {$t$};
\node at (-.5,9) {$x$};
\node at (11,-.5) {$z$};
\draw[,very thick] (4,0)--(0,8);
\foreach \x/\xtext in {(8, 2), (6, 4), (4, 6), (9, 0), (7, 2), (5, 4), (3, 6), (6, 2), (4, 4),
(2, 6), (0, 8), (7, 0), (5, 2), (3, 4), (1, 6), (6, 0), (4, 2), (2, 4),
(5, 0), (3, 2), (4, 0),(10,0)}
\node[vertice] at \x {};
\end{tikzpicture}
\caption{Newton polygon of $f_x$}
\label{fig:newtonfxz}
\end{center}
\end{subfigure}
\begin{subfigure}[b]{.55\textwidth}
\begin{center}
\begin{tikzpicture}[scale=.5,xscale=.8,vertice/.style={draw,circle,fill,minimum size=0.15cm,inner sep=0}]
\draw[fill, color=blue!10!white] (13,0)--  (8,0)--(0,4)--(0,7)--(13,7);
\draw[help lines,step=1,dashed] (0,0) grid (13,7);
\draw[very thick] (0,7)--(0,0)--(13,0);
\foreach \x/\xtext in {0,...,6}
 \node[anchor=east] at (-.1,\x) {\tiny{\x}};
\foreach \y in {0,1,...,12}
 \node[anchor=north] at (\y,-.1) {\tiny{\y}};
\node[below] at (12,-.6) {$t$};
\node at (-.5,7) {$x_1$};
\node at (13,-.5) {$\bar{z}_1$};
\draw[,very thick] (8,0)--(0,4);
\foreach \x/\xtext in {(10, 6), (10, 5), (10, 4), (8, 6), (10, 3), (8, 5), (10, 2), (8, 4),
(6, 6), (10, 1), (8, 3), (6, 5), (10, 0), (8, 2), (6, 4), (8, 1), (6,
3), (4, 5), (8, 0), (6, 2), (4, 4), (6, 1), (4, 3), (4, 2), (2, 4), (2,
3), (0, 4),(12, 6), (12, 5), (12, 4), (12, 3), (12, 2), (12, 1), (12, 0)}
\node[vertice] at \x {};
\end{tikzpicture}
\caption{Newton polygon of $f_{x,1}$}
\label{fig:newtonfxz1}
\end{center}
\end{subfigure}
\caption{}
\end{center}
\end{figure}

The other special pencil at $P_1$ is given by 
\begin{gather*}
f_x(z,x)-Tz^{10}=-\frac{25}{4}z^4x^6-\frac{17}{4}z^3x^6+4z^2x^6-4x^6z-12z^3x^4
+6x^4z^2\\-\frac{71}{4}z^6x^4+4z^5x^4+\frac{75}{4}z^4x^4-\frac{25}{2}z^8x^2+\frac{75}{4}z^7x^2
+\frac{59}{4}z^6x^2-\frac{99}{4}z^5x^2\\+12x^2z^4-4x^2z^3+z^4-4z^5+\frac{41}{4}z^6-\frac{25}{2}z^7+\frac{25}{4}z^9+x^8-Tz^{10}.
\end{gather*}
Let us check that this special pencil  has  $4$ branches for all $t\in \mathbb{K}$
and one dicritical $E.$
Its Newton polygon is in Figure~\ref{fig:newtonfxz}; there is only one edge~$\ell$, which is not dicritical
and the quasihomogenous polynomial associated to the edge is $P_\ell=(x^2-z)^4$.
We need only one toric-Newton transformation at this stage:
$$
\varphi_M(z_1,x_1)=(z_1^2 x_1,z_1 x_1),\quad 
x_1\mapsto \bar{x}_1+1
$$
The Newton polygon of the strict transform $f_{x,1}(z_1,\bar{x}_1)$ is in Figure~\ref{fig:newtonfxz1}.
We have only one edge $\ell_1$, which is non-dicritical with
$P_{\ell_1}=(\bar{x}_1+z_1^2)^4$. If we perform the translation of Remark~\ref{obs:trans}
we obtain a new special pencil $f_{x,2}(z_2,x_2)$. The Newton polygon
is in Figure~\ref{fig:newtonfxz2}. We have only one edge $\ell_2$, which is dicritical, since
$$
P_{\ell_2}=x_2^4- 2 x_2^2 z_2^5+(1-T) z_2^{12},
$$
i.e., 
its v-ratio equals~$2$, $q_\ell(z)=z^4 -2 z^2+1-T$ and
$q_E(z)=z^4 -2 z^2+1$.
The roots of   $q_E'(z)$ are $\alpha=0,1,-1$, hence the bound equals~$3$. 
Since $q_E(0)=0$ and $q_E(1)=q_E(-1)=1$, there are exactly two atypical values, $t=0,1$.  

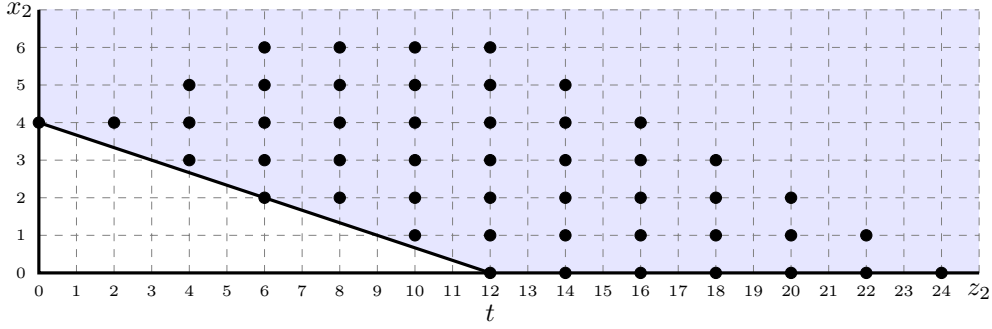
\begin{figure}[ht]
\begin{center}
\begin{tikzpicture}[scale=.5,xscale=1,vertice/.style={draw,circle,fill,minimum size=0.15cm,inner sep=0}]
\draw[fill, color=blue!10!white] (25,0)-- (12,0)--(0,4)--(0,7)--(25,7);
\draw[help lines,step=1,dashed] (0,0) grid (25,7);
\draw[very thick] (0,7)--(0,0)--(25,0);
\foreach \x/\xtext in {0,...,6}
 \node[anchor=east] at (-.1,\x) {\tiny{\x}};
\foreach \y in {0,1,...,24}
 \node[anchor=north] at (\y,-.1) {\tiny{\y}};
\node[below] at (12,-.6) {$t$};
\node at (-.5,7) {$x_2$};
\node at (25,-.5) {$z_2$};
\draw[,very thick] (12,0)--(0,4);
\foreach \x/\xtext in {(18, 0), (12, 1), (6, 6), (14, 4), (10, 6), (16, 2), (12, 5), (6, 2),
(14, 0), (24, 0), (10, 3), (18, 1), (10, 4), (12, 2), (14, 5), (8, 5),
(16, 3), (4, 4), (12, 6), (20, 0), (6, 3), (14, 1), (0, 4), (8, 6), (12,
3), (6, 4), (18, 2), (8, 2), (16, 4), (4, 5), (20, 1), (10, 5), (16, 0),
(10, 1), (14, 2), (22, 0), (6, 5), (8, 3), (18, 3), (12, 0), (20, 2),
(16, 1), (12, 4), (4, 3), (14, 3), (22, 1), (10, 2), (2, 4), (8, 4)}
\node[vertice] at \x {};
\end{tikzpicture}
\caption{Newton polygon of $f_{x,2}$}
\label{fig:newtonfxz2}
\end{center}
\end{figure}

\end{ejm}

\begin{ejm}\label{ejm:rfr}
The referee asked   whether Gwo{\'z}dziewicz's question
has an affirmative answer for other positive integer $n \geq 3$.
In this example we provide a polynomial family which confirms the required positive answer.

For any~$d$, we consider two monic polynomials $q(t),Q(t)\in\mathbb{K}[t]$ of degrees~$2 d$ and $2 d+1$, respectively,
such that:
\begin{enumerate}
\enet{(C\arabic{enumi})}
\item\label{c1} $\deg(Q(t)-t q(t))\leq d$.
\item\label{c2} $q(t)=\prod_{j=1}^m (t-a_j)^{m_j}$, $\sum_{j=1}^m m_j=2 d$, $m_j\geq 2$.
\item\label{c3} $Q(t)=\prod_{j=1}^n (t-b_j)^{n_j}$, $\sum_{j=1}^n n_j=2 d+1$, $n_j\geq 2$.
\end{enumerate}
Let $f(x,y)$ be the polynomial
$$
f(x,y)=(y+1)\left(x q(x y)+(y+1) Q(x y)\right).
$$
Its Newton polygon has four  edges whose vertices are
given by 
$${[(0,0),(0,1),(2d+1,2d+2),(2d+1,2d),(1,0)]}.$$
Let $\ell_1=[(0,1),(2d+1,2d+2)]$,  $\ell_2=[(1,0),(2d+1,2d)]$ and
$\ell_3=[(2d+1,2d),(2d+1,2d+2)]$ be the edges not passing through the origin.

The support polynomial $f_{\ell_1}$ is $y Q(x y)$. Because of condition~\ref{c3},
one can see that each root~$b_j$ induces a dicritical section with bamboo,
producing exactly one atypical value.

The support polynomial $f_{\ell_2}$ is $x q(x y)$. As above, condition~\ref{c2} implies
that each root~$a_j$ induces a dicritical section with bamboo,
producing exactly one atypical value.

The support polynomial of the vertical edge $\ell_3$ is $f_{\ell_3}=x^{2 d+1} y^{2 d}(y+1)^2$.
The condition~\ref{c1} implies that the translation $y=y_1-1$ produces a new edge $\ell'_3=[(0,0),(2d+1, 2)]$. 
Hence $ \ell_3'$ is a dicritical edge with bamboo ($v_{\ell'_3}=2$) and only one atypical value. 

Of course the conditions \ref{c1}, \ref{c2} and \ref{c3} impose restrictions but one can see that
solutions exist and for generic choices, the atypical values for each dicritical are distinct, providing
the required affirmative answer.

For example, this is the case if 
$$
Q(t)=(t+1)^{2 d+1}
\text{ and } 
p(t)=\left(\prod_{j=1}^d (t-a_j)\right)^2.
$$
Then~\ref{c1} allows to give the coefficients of $p(t)$. 
A tedious verification ensures that 
$f$ has $d+2$ dicritical sections with bamboo, 
and the generic fiber has $d+2$ branches at infinity, its genus being~$d$. 
The fact that they have $d+2$ different atypical values has been checked for
small values of~$d$ ($\leq 20$) with \texttt{SAGE}~\cite{sage510}.
\end{ejm}

\begin{obs}
Note that for $d=1$, we can obtain a polynomial with three branches and degree~$7$,
while Example~\ref{ejm-cotaep}, with two branches, has degree~$11$. Surprisingly,
this is the smallest degree for a two-branch polynomial reaching the bound. Note
that all the examples have only dicritical sections.
\end{obs}

\begin{ejm}
Both Examples~\ref{ejm-cotaep} and~\ref{ejm:rfr} have only dicritical sections. We have found
also an example of degree~$18$, with two dicriticals (with bamboo), one of them~$E$ with multiplicity~$2$,
hence having also three branches at infinity for the generic fiber and three atypical values.
The fiber corresponding to the value in $A_E^*$ has only two branches at infinity, i.e., 
$\nu_{\infty}^{\text{\rm min}}+1=\nu_\infty^{\text{\rm gen}}$, check the bounds in Page~\pageref{numin}.
\end{ejm}

\section*{Acknowledgements}
First author is partially supported by 
MTM2010-21740-C02-02; the last two authors are partially
supported by the grant 
MTM2010-21740-C02-01. The authors would like to thank Pierrette Cassou-Nogu{\`e}s and  referees for useful remarks and 
valuable comments.

\appendix
\section{Hensel's Lemma}\label{sec:hensel}
In order to be clear which flavor of Hensel's Lemma we are
going to use, we state and prove the following elementary result

Let $\mathbb{K}$ be a field and fix a weight $\omega(x,y):=n x+ m y$
for $n,m\in\mathbb{N}$. Given $0\neq F\in\mathbb{K}[[x,y]]$, we will
consider its decomposition in $\omega$-quasihomogeneous forms
\begin{equation}\label{eq:wdesc}
F(x,y)=F_{a+b}(x,y)+F_{a+b+1}(x,y)+\dots,
\end{equation}
where the subindex means the $\omega$-weight

\begin{lema}[Hensel's Lemma]
Asume that $F_{a+b}(x,y)=f_a(x,y) g_b(x,y)$, $f_a,g_b\in\mathbb{K}[x,y]$
quasihomogeneous and coprime.h
Then, there exist
$$
f,g\in\mathbb{K}[[X,Y]],\quad
f=f_a+f_{a+1}+\dots,\quad 
g=g_b+g_{b+1}+\dots
$$
such that $F=f g$. Moreover if $f_a$ is an irreducible polynomial, then $f$ is an irreducible power series.
\end{lema}

\begin{proof}
We need to find recursively $\omega$-quasihomogeneous
polynomials $f_{a+k},g_{b+k}$, $k\in\mathbb{N}$ such that
\begin{equation}\label{eq:rec}
f_a(x,y) g_{b+k}(x,y)+g_b(x,y) f_{a+k}(x,y)= F_{a+b+k}^*(x,y)
\end{equation}
where $g_{b+k},f_{a+k}$ are the unknowns and $F_{a+b+k}^*$ is obtained
from $F_{a+b+k}$ and the previous solutions for $k'<k$. 

Let us decompose the above polynomials (where now the subindex correspond now to the homogeneous degree
for the weigh $\omega_0$ defined by $n=m=1$):
$$
\begin{matrix}
f_a(x,y)=&x^{a_x} y^{a_y} f_{a'}(x^m,y^n),&\quad&a=&n a_x+ m a_y+a' m n\\
g_b(x,y)=&x^{b_x} y^{b_y} g_{b'}(x^m,y^n),&\quad&b=&n b_x+ m b_y+b' m n\\
f_{a+k}(x,y)=&x^{c_x} y^{c_y} \tilde{f}_{c}(x^m,y^n),&\quad&a+k=&n c_x+ m c_y+c m n\\
g_{b+k}(x,y)=&x^{d_x} y^{d_y} \tilde{g}_{d}(x^m,y^n),&\quad&b+k=&n d_x+ m d_y+d m n\\
F_{a+b+k}^*(x,y)=&x^{e_x} y^{e_y} \tilde{F}_{e}(x^m,y^n),&\quad&a+b+k=&n e_x+ m e_y+e m n.
\end{matrix}
$$
The decompositions of $a,b,c,d,e$ are unique if we assume that the all indices are non-negative,
the coefficient of $n$ is less than $m$ and the coefficient of $m$ is less than $n$.
We prove it in several steps.

\begin{clm}\label{clm0}
The statement holds for $\omega_0$, i.e., the homogeneous case.
\end{clm}
It is an immediate consequence of the properties of the resultant.

\begin{clm}\label{clm1}
The statement holds if $f_a(x,y)$ is a power of~$x$ or~$y$.
\end{clm}

Assume that $f_a$ is a power of $x$. In this case, we have
\begin{itemize}
\item $a=n (a_x+m a')$, $0\leq a_x<m$.
\item $g_b(0,1)\neq 0$, i.e., $b_x=0$.
\end{itemize}
The following equalities hold:
$$
n(a_x+d_x)+m d_y+(a'+d) m n= n c_x+m (b_y+d_y)+ (b'+c) m n=n e_x+ m e_y+e m n.
$$
We deduce that $e_x=c_x=a_x+d_x-\alpha m$, $e_y=d_y=b_y+d_y-\beta n$,
where $\alpha,\beta\in\{0,1\}$ and
$$
e=a'+d+\alpha=b'+c+\beta.
$$
Equation \eqref{eq:rec} is equivalent to
$$
x^{\alpha+a'} \tilde{g}_{d}(x,y)+y^\beta \tilde{g}_{b'}(x,y) \tilde{f}_{c}(x,y)= \tilde{F}_{e}(x,y),
$$
which follows from Claim~\ref{clm0}, and Claim~\ref{clm1} holds.

\begin{clm}\label{clm2}
The statement holds if both $f_a$ and $g_b$ are coprime with $x,y$.
\end{clm}
In this case $a_x=a_y=b_x=b_y=0$ and
$$
d_x=e_x,\quad d_y=e_y,\quad
a'+d=b'+c=e.
$$
Hence \eqref{eq:rec} is transformed again in its homogeneous version
and Claim~\ref{clm2} follows from again from Claim~\ref{clm0}.
Combining these claims, the statement is proved.
\end{proof}


\def\cprime{$'$}
\providecommand{\bysame}{\leavevmode\hbox to3em{\hrulefill}\thinspace}
\providecommand{\MR}{\relax\ifhmode\unskip\space\fi MR }
\providecommand{\MRhref}[2]{%
  \href{http://www.ams.org/mathscinet-getitem?mr=#1}{#2}
}
\providecommand{\href}[2]{#2}

\end{document}